\documentclass[11pt]{article}
\usepackage{amsmath,amsthm,amssymb,amsfonts,amsbsy,mathtools,enumerate}
\usepackage{tikz}
\usepackage{float}
\usepackage{graphicx}
\usepackage[dvips]{epsfig}
\usepackage{blkarray}
\usepackage{caption, subcaption}
\usepackage{nicematrix}
\usepackage[toc,page]{appendix}
\usepackage{hyperref}
\usepackage{url}
\usepackage{quiver}
\usepackage{tikz-cd}

\theoremstyle{plain}
\newtheorem{theorem}{Theorem}[section]
\newtheorem{corollary}[theorem]{Corollary}

\newtheorem{lemma}[theorem]{Lemma}
\newtheorem{prop}[theorem]{Proposition}

\newtheorem{defn}[theorem]{Definition}

\newcommand\sqr[2]{{\vbox{\hrule height.#2pt
    \hbox{\vrule width.#2pt height#1pt \kern#1pt
        \vrule width.#2pt}\hrule height.#2pt}}}
\newcommand\qedbox{%
	\ifmmode\eqno\sqr53
	\else\nolinebreak\ \hfill\sqr53\medbreak\fi}

\newcommand\Zv{{\mathbf v}}

\newcommand\Zy{{\mathbf y}}

\def\N{{\mathbb N}}

\newcommand\Mat[3]{\text{Mat}_{#1 \times #2}(#3)}
\newcommand\re{{\mathbb R}}
\newcommand\fld{{\mathbb F}}    
\DeclareMathOperator{\GQ}{GQ}

\DeclareMathOperator\im{im}

\newcommand{\sgn}{\rm sgn}
\newcommand\dotcup{\raisebox{-2pt}{\scalebox{1.3}{$\dot{\cup}\,$}}}

\tikzcdset{scale cd/.style={every label/.append style={scale=#1},
    cells={nodes={scale=#1}}}}

\begin{document}
\title{Clique complexes of strongly regular graphs, their eigenvalues, and cohomology groups}
\author{Sebastian M. Cioab\u{a}\footnote{Department of Mathematical Sciences, University of Delaware, Newark, DE 19716-2553, USA, {\tt cioaba@udel.edu}. This research has been partially supported by NSF grant DMS-2245556.}\, , Krystal Guo\footnote{Korteweg-de Vries Institute, University of Amsterdam, Amsterdam, The Netherlands, {\tt k.guo@uva.nl}}\, , Chunxu Ji\footnote{Department of Mathematics, California State University San Marcos, San Marcos, CA 92096. This research was done while the author was an undergraduate student at the University of Delaware as part of his senior thesis.}, and Mutasim Mim\footnote{Department of Mathematics, The Graduate School and University Center of The City University of New York, New York, NY 10016, USA, {\tt mmim@gradcenter.cuny.edu}. Part of this research was done while the author was a graduate student in the Department of Mathematical Sciences at the University of Delaware.}}
\date{\today}
\maketitle

\begin{abstract}
It is known that non-isomorphic strongly regular graphs with the same parameters must be cospectral (have the same eigenvalues). In this paper, we investigate whether the spectra of higher order Laplacians associated with these graphs can distinguish them. In this direction, we study the clique complexes of strongly regular graphs, and determine the spectra of the triangle complexes of several families of strongly regular graphs including Hamming graphs and Triangular graphs. In many cases, the spectrum of the triangle complex distinguishes between strongly regular graphs with the same parameters, but we find some examples where that is not the case. We also prove that if a graph has the property that for any induced cycle, there are four consecutive vertices on the cycle with a common neighbor, then the first cohomology group of the graph is trivial and we apply this result to several families of graphs.


\end{abstract}

\section{Introduction}




In the area of spectral characterization of graphs, one of the main goals is determining whether the eigenvalues of a given graph matrix (the adjacency matrix or the Laplacian matrix, for example) {\em characterize} the graph, see \cite{VanDamHaemersSurvey1, VanDamHaemersSurvey2}. It is well known that there exist non-isomorphic graphs that are cospectral (have the same eigenvalues). In such cases, it is of interest to determine if the eigenvalues of other graph matrices can distinguish these graphs. In that direction, there has been much interest in  exploring the relationship between the spectra of various matrices of graph and graph isomorphism, especially for the class of strongly regular graphs, see \cite{BabaiAutSRG, SpielmanSRG} for example. 

A graph is strongly regular with parameters $(v,k,\lambda,\mu)$ or a $(v,k,\lambda,\mu)$-SRG for short, if it has $v$ vertices, is $k$-regular, any two adjacent vertices have exactly $\lambda$ common neighbors and any two distinct and non-adjacent vertices have exactly $\mu$ common neighbors. We quote below from the preface of the recent monograph \cite{BM}:
{\em The topic of strongly regular graphs is an area where statistics, Euclidean
geometry, group theory, finite geometry, and extremal combinatorics meet. The subject concerns beautifully regular structures, studied mostly using spectral methods, group theory, geometry and sometimes lattice theory.}

In general, there is no relation between the eigenvalues of the adjacency matrix and the eigenvalues of the Laplacian matrix of a graph. However, in the case of regular graphs, such a connection between these spectra exists and the eigenvalues of one matrix can be deduced easily from the eigenvalues of the other matrix. Non-isomorphic strongly regular graphs with the same parameters must be cospectral (have the same adjacency or Laplacian matrix eigenvalues), see \cite[Thm. 9.1.3]{BH}. One class of matrices which has been studied are the symmetric powers of graphs. In \cite{AudGodRoyRud2006}, the authors show that the spectra of the symmetric square of strongly regular graphs with the same parameters are equal, while giving computational evidence that higher symmetric powers may be possible graph invariants for strongly regular graphs. In \cite{AlzIglPig2010, BarPon2009}, it is shown that there are infinite families of pairs of non-isomorphic graphs with cospectral $m$-th symmetric powers, for all $m$, however none of these graphs are strongly regular. Similarly, the spectrum of a matrix related to a quantum walk on a strongly regular graph was proposed in \cite{EHSW06,ESWH} and the first strongly regular counterexamples were found in \cite{GodGuoMyk2017}.

In this paper, we study the spectra of the up-Laplacians $L_i^{\uparrow}$  and first cohomology groups (see Section \ref{sec:defn} for the definitions) of the clique complexes of various classes of strongly regular graphs. These clique complexes are also of interest in topological data analysis, see \cite{Was2018}. 

In Section \ref{sec:defn}, we present the main definitions and notations. In Section \ref{sec:PG-GQ}, we determine the spectra of the up-Laplacian $L_1^\uparrow$ of partial geometries and generalized quadrangles. In Section \ref{sec:familes}, we determine the spectra of the up-Laplacian $L_1^\uparrow$ of Hamming graphs (Section \ref{thm:Hqn}),  and most of this section is devoted to determining the $L_1^{\uparrow}$ spectra of the Triangular graphs (Theorem \ref{thm:Tn}). The eigenvalues of the higher dimensional up-Laplacians $L_i^{\uparrow}$ of the Triangular graphs are also obtained there. In Section \ref{subsec:Tch}, we show that if a graph has the property that for any induced cycle, there are four consecutive vertices on the cycle with a common neighbor, then the first cohomology group of the graph is trivial and we apply this result to several families of graphs. 

The focus then shifts to describing our computations of the spectra of the up-Laplacians $L_1^{\uparrow}$ of small strongly regular graphs in Section \ref{sec:srgs}, with examples of cospectral pairs given in Section \ref{sec:cosp} including a pair of strongly regular graphs on 40 vertices. 


\section{Definitions}\label{sec:defn}

In this section, we introduce the definitions of simplicial complexes and some of their properties. Our main reference is \cite{BGP}. If $V$ is a finite set, then a \textit{simplicial complex} $X$ on $V$ is a collection of subsets of $V$ such that if $F\in X$, then any subset of $F$ is also in $X$. We call such $F$ a \textit{face} of $X$, and the \textit{dimension} of $F$ is defined to be $|F|-1$. The set of $i$-dimensional faces in $X$ is denoted by $X_i$, and conventionally $X_{-1} =\{\emptyset\}$. Assume we have a global ordering on the ground set $V$. If $F=\{x_0,x_1,\dots,x_i\}$, where $x_0<x_1<\dots<x_i$, then we define
\begin{equation}\label{eq:signing}    [F:K]=
\begin{cases}
        (-1)^j, & \text{if }K\subset F \text{ and } F\setminus K = \{x_j\},\\
        0, & \text{otherwise. }
\end{cases}
\end{equation}

In this paper, we will consider the clique complex of a graph; the \textit{clique complex $\mathcal{K}(G)$ of an undirected graph $G$} is the simplicial complex formed by the subsets of vertices in the cliques of $G$. Here, the ground set is the vertex set of a graph, equipped with a total ordering. If we restrict to faces of dimension $2$, the \textit{triangle complex $\mathcal{T}(G)$ of $G$} is the simplicial complex formed by the subsets of vertices in the triangles of $G$. 

Figure \ref{fig:simp-complex} shows the clique complex of a graph $G$ (which the complete graph $K_4$ with one edge deleted) with the incidences between its faces being decorated by the signing given by \eqref{eq:signing}. We will use this graph as a running example throughout this section. 

\begin{figure}[htbp]
    \centering
    \includegraphics[scale=0.8]{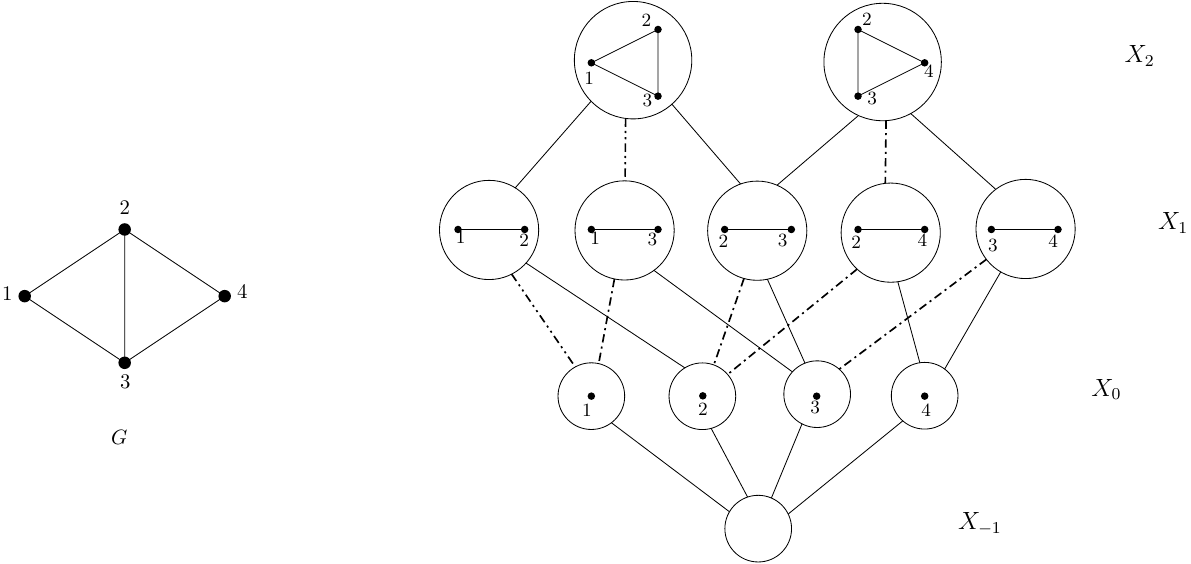}
    \caption{A graph $G$, \textit{left}, and its clique complex, \textit{right}. In the clique complex, the lines decorated with dots and dashes connect faces $F,K$ where $[F:K] = -1$ and the other lines connect faces $F,K$ where $[F:K] = 1$, where we take the natural ordering of the vertices, $1<2<3<4$.       \label{fig:simp-complex}}
    \label{graph1}
\end{figure}

Since we are interested in clique complexes of finite graphs, we will give simplified, finite-dimensional versions of the definitions, for readability. Note that we use $\Mat{U}{W}{\fld}$ to denote matrices with rows indexed by set $U$, columns indexed by $W$ and elements in field $\fld$. For $i\geq 0$, we use the notation $\re^{X_i}$ for the vector space $\{f:X_i\rightarrow \re\}$. The elements
of $\re^{X_i}$ are called $i$-dimensional cochains of $X$ with coefficients in $\re$. Biggs \cite[Def 4.1]{Biggs} calls $\re^{X_0}$ the vertex-space of $G$ and $\re^{X_1}$ the edge-space of $G$.

The \textit{coboundary map} $\delta_i \in \Mat{X_{i+1}}{X_i}{\re}$ is given by 
\[
(\delta_i)_{H,F} = [H:F],
\]
for $H\in X_{i+1}$ and $F \in X_{i}$. We can think of $\delta_i$ as an incidence matrix between the $(i+1)$-dimensional faces and the $i$-dimensional faces of the simplex, with a signing given by \eqref{eq:signing}. Analogously, the \textit{boundary map} $\partial_i \in \Mat{X_{i-1}}{X_i}{\re}$ is given by $\partial_{i} = \delta_{i-1}^T$.
To help keep track of these incidence matrices, we use the following diagram:
\[
\re^{X_{i+1}} \xleftrightharpoons[\partial_{i+1}]{\delta_i} \re^{X_i} \xleftrightharpoons[\partial_i]{\delta_{i-1}} \re^{X_{i-1}} .
\]

For the clique complex of any graph, the coboundary map $\delta_0$ is a signing of the usual edge-vertex incidence matrix of the graph. For the graph in  Figure \ref{fig:simp-complex}, we write $\delta_0$ and $\delta_1$ below: 
\[
\delta_0 = \begin{bmatrix}
-1 & 1 & 0 & 0 \\
-1 & 0 & 1 & 0 \\
0 & -1 & 1 & 0 \\
0 & -1 & 0 & 1 \\
0 & 0 & -1 & 1
\end{bmatrix}, \quad
\delta_1 = \begin{bmatrix}
1 & -1 & 1 & 0 &0 \\
0&0 &1 &-1 & 1 
\end{bmatrix},
\]
where the ordering of the rows and columns are given by the left to right ordering in Figure \ref{fig:simp-complex}. 



We define the \textit{Down Laplacian} $L_i^\downarrow$, the \textit{Up Laplacian} $L_i^\uparrow $, and the \textit{Total Laplacian} $L_i$ by
\[
L_i^\downarrow := \delta_{i-1}\delta_{i-1}^T,\ \ \ \ L_i^\uparrow:= \delta_{i}^T\delta_i,\ \ \ \ L_i := L_i^\downarrow+L_i^\uparrow.\]
For the clique complex of any graph, the up-Laplacian $L_0^\uparrow$ coincides with the usual Laplacian of the graph while the Down Laplacian $L_0^\downarrow$ is the all one $X_0\times X_0$ matrix $J$. In this setting, we will refer to $L_1^\uparrow(G)$ as the \textit{up-Laplacian matrix associated with the graph $G$}, and we write $L_1^\uparrow$ when the context is clear. Similarly, we will refer to $L_1^\downarrow(G)$ as the \textit{down Laplacian matrix associated with the graph $G$}, and we write $L_1^\downarrow$ when the context is clear.
For the graph in Figure \ref{fig:simp-complex}, these matrices are:
\[
{L_1^\downarrow = \begin{bmatrix}
2 & 1 & -1 & -1 & 0 \\
1 & 2 & 1 & 0 & -1 \\
-1 & 1 & 2 & 1 & -1 \\
-1 & 0 & 1 & 2 & 1 \\
0 & -1 & -1 & 1 & 2
\end{bmatrix}, \quad
L_1^\uparrow = \begin{bmatrix}
1 & -1 & 1 & 0 & 0 \\
-1 & 1 & -1 & 0 & 0 \\
1 & -1 & 2 & -1 & 1 \\
0 & 0 & -1 & 1 & -1 \\
0 & 0 & 1 & -1 & 1
\end{bmatrix}, \quad
L_1 = \begin{bmatrix}
3 & 0 & 0 & -1 & 0 \\
0 & 3 & 0 & 0 & -1 \\
0 & 0 & 4 & 0 & 0 \\
-1 & 0 & 0 & 3 & 0 \\
0 & -1 & 0 & 0 & 3
\end{bmatrix}. }
\]


The following results describe some basic properties of these matrices (see \cite{BGP} for proofs and more details).
\begin{lemma}\label{lem:basics} For any $i\geq 1$,
the following hold 
\begin{enumerate}[(i)]
\item $\delta_i \circ \delta_{i-1} = 0$ and equivalently, $\im \delta_{i-1}\subseteq\ker\delta_i$;
\item $L_i^\uparrow L_i^\downarrow = L_i^\downarrow L_i^\uparrow = 0$;
\item $\delta_{i-1}^T \circ \delta_i^T = 0;$ and
\item $\ker L_i^\downarrow = \ker\partial_i,$ $\im L_i^\downarrow = \im \delta_{i-1},$ $\ker L_i^\uparrow = \ker \delta_i,$ $\im L_i^\uparrow = \im\partial_{i+1}$. \qedbox
\end{enumerate}
\end{lemma}
For $j\geq 1$, the $j$-th cohomology group is defined as $H^j=\ker(\delta_j)/\im(\delta_{j-1})$ for $j\geq 1$. Biggs \cite{Biggs} refers to $\ker \delta_0$ as the cycle space of the graph $G$ (see also Section \ref{subsec:Tch}) and to $\ker\partial_1=\ker(\delta_0^T)$ as the cut space of $G$. Using the previous proposition, one can show that the non-zero spectrum of $L_i$ is the union of the non-zero spectrum of $L_i^{\uparrow}$ and the non-zero spectrum of $L_i^{\downarrow}$, see also \cite{MullasHorakJost}.

Similar to the notion of a vertex degree, one can define the degree of an arbitrary face as follows.
\begin{defn}
The number of $(i+1)$-faces of $X$ that contain a given $i$-dimensional face $F$ is called the \textit{degree} of $F$, and is denoted by $\deg(F)$.
\end{defn}
\noindent Moreover, given $(F,F')\in X_i^2$, define
\begin{equation}\label{eq:epsilon}    \epsilon_{F,F'} =
\begin{cases}
        [F:F\cap F'][F':F\cap F'] & \text{if }|F\cap F'| = i,\\
        0 & \text{otherwise. }
\end{cases}
\end{equation}

\begin{lemma}\label{lem:epsLdown}
If $(F,F')\in X_i^2$ such that $F\cup F'\in X_{i+1}$, then
\begin{equation}\label{eq:epsFF'}
\epsilon_{F,F'}=[F:F\cap F'][F':F\cap F'] = -[F\cup F':F][F\cup F':F'].
\end{equation}
Furthermore, we can express the matrix form of $L_i^\downarrow$ and the matrix form of $L_i^\uparrow$ entry-wise as follows:
\begin{equation*}    (L_i^\downarrow)_{F,F'} =
\begin{cases}
        i+1, & \text{if }F = F',\\
        \epsilon_{F,F'}, & \text{if } |F\cup F'| = i+2,\\
        0, & \text{otherwise. }
\end{cases}
\end{equation*}
and
\begin{equation*}    (L_i^\uparrow)_{F,F'} =
\begin{cases}
        \deg(F), & \text{if }F = F',\\
        -\epsilon_{F,F'}, & \text{if } F\cup F'\in X_{i+1}\\
        0, & \text{otherwise. }
\end{cases} \qedbox
\end{equation*}
\end{lemma}

Note that the spectra of the up and down Laplacians do not depend on the choice of ordering of the vertices and we include a proof of this fact in Appendix  \ref{a:inv} for completeness.


Recall the following linear algebraic property, which can be found in standard texts (see \cite[Theorem 1.3.20]{HJ} for example). 

\begin{lemma}\label{nonzero}
If A and B are $m\times n$ real matrices, then the characteristic polynomials of $AB^T$ and $B^T A$ differ by a factor of $x^{m-n}$. Consequently,  $AB^T$ and $B^T A$ have the same non-zero eigenvalues, with the same multiplicities. \qedbox 
\end{lemma}

Applying Lemma \ref{nonzero} to the down Laplacian, we obtain the following. 

\begin{lemma}\label{down}
Given a graph $G$, the non-zero eigenvalues of down Laplacian $L_1^\downarrow(G)$ (with their multiplicities) and the non-zero eigenvalues of ordinary Laplacian $L(G)$ (with their multiplicities) are the same.
\end{lemma}
\begin{proof}
Because $L_1^{\downarrow}=\delta_0\delta_0^{T}$ and $L(G)=L_0^{\uparrow}=\delta_0^T\delta_0$, Lemma \ref{nonzero} implies the desired result.
\end{proof}

\begin{prop}
If $G$ is a simple graph, then $\dim(\ker(L_1^\uparrow(G))\geq |V(G)|-1$.
\end{prop}
\begin{proof}
Because $G$ is connected, $\ker\delta_0$ is the $1$-dimensional subspace of $\re^{V(G)}$ spanned by the all one vector (see \cite[Ch. 4]{Biggs} for example). By the Rank-Nullity theorem, $\dim(\text{im }\delta_0)=|V(G)|-1$. Lemma \ref{lem:basics} implies that $\text{im }\delta_0\subseteq\ker\delta_1=\ker L_1^{\uparrow}(G)$ and therefore, $|V(G)|-1=\dim(\im(\delta_0))\leq \dim(\ker(L_1^\uparrow(G)))$.
\end{proof}
If the distinct eigenvalues of a real and symmetric matrix $M$ are $\theta_1>\theta_2>\ldots > \theta_d$ with respective multiplicities $m_1,\ldots,m_d$, we will often write the spectrum of $M$ in the following array:
\[
\begin{pmatrix}
\theta_d & \cdots &\theta_1\\
m_d& \cdots & m_1
\end{pmatrix}. \]

We will use the proposition below regarding the spectrum of up-Laplacian on $K_n^{k-1}$, the clique complex on $[n]$ with all subsets of size at most $k$. This result has been obtained in different contexts, see \cite[Lemma 8]{GW} or \cite[Example 2.3]{BGP}.
\begin{prop}[\cite{BGP,GW}]\label{prop:kn-eigs}
    For all $i\leq k$, the spectrum of $L^{\uparrow}_{i}(K_n^{k})$ is 
    \[
        \begin{pmatrix}
        0 & n\\
        \binom{n-1}{i} & \binom{n-1}{i+1}
        \end{pmatrix}.
    \]    
\end{prop}

\section{Partial geometries and generalized quadrangles} \label{sec:PG-GQ}

\begin{lemma}\label{lem:uniquefaces}
If $S$ is a simplicial complex such that each $i$-dimensional face lies on a unique $(i+1)$-dimensional face, then the spectrum of $L_i^\uparrow$ is
\[
\begin{pmatrix}
0 & i+2\\
(i+1)|X_{i+1}| & |X_{i+1}|
\end{pmatrix}.
\]
\end{lemma}

\begin{proof}
We partition the $i$-dimensional faces of $S$ into $|X_{i+1}|$ parts: $X_i = \dotcup_{H \in X_{i+1}} Y_H$, 
where $Y_H=\{ F \in X_i \mid F \subset H\}$ is the set of $i$-dimensional faces contained in an $(i+1)$-dimensional face $H$. 
For $H=\{x_0,x_1,\ldots,x_{i+1}\} \in X_{i+1}$ with $x_0 < x_1 < \ldots <x_{i+1}$, we have that
\[
Y_H=\{H \setminus \{x_0\}, H \setminus \{x_1\}, \ldots, H \setminus \{x_i\}\}.
\]
Let $\Zy$ of the $(i+2)$-dimensional vector whose entries alternate between $1$ and $-1$, that is 
\[\Zy^T = \begin{pmatrix}
1 & -1 & 1 & \cdots & (-1)^{i+1}
\end{pmatrix}.
\]
For some ordering of the elements of $X_i$, we have that 
\[
\delta_i = I_{|X_{i+1}|} \otimes \Zy^T
\]
where $I_{|X_{i+1}|}$ is the $|X_{i+1}|\times |X_{i+1}|$ identity matrix. Thus, 
\[
L_i^\uparrow = \delta_i^T \delta_i = I_{|X_{i+1}|} \otimes \Zy\Zy^T.
\]
Since $\Zy\Zy^T$ is the all one matrix of order $i+2$ whose only non-zero eigenvalue is $i+2$ with eigenvector $\Zy$, the eigenvalues of $L_i^\uparrow$ are as claimed.
\end{proof}

A partial geometry $pg(K,R,T)$ is an incidence structure of points and lines with the following properties:
\begin{enumerate}
    \item Any line has $K$ points and any point is contained in $R$ lines.
    \item Any two points lie on at most one line.
    \item For any point $p$ and any line $L$ such that $p$ is not on $L$, there are exactly $T$ lines containing $p$ that intersect $L$.
\end{enumerate}
The point graph of the geometry has the points as vertices and two points are adjacent if they are contained in a line. This graph is strongly regular (see \cite[Chapter 21]{vLW} for more details on partial geometries and strongly regular graphs). When $T=1$, the partial geometry is called a generalized quadrangle and $GQ(K-1,R-1)$ is commonly used to denote $pg(K,R,1)$. 

In the point graph of $\GQ(s,t)$, the maximum cliques come from lines and are of order $s+1$. If adjacent vertices $u,v$ have a common neighbour $w$, then $u,v,w$ must all lie on some line. Thus, every $3$-clique is contained in a unique clique of maximum size.
Using Proposition \ref{prop:kn-eigs}, we determine the spectrum of many up-Laplacians of generalized quadrangles, as follows.

\begin{lemma}\label{lem:gq}
For the point graph of generalized quadrangle $\GQ(s,t)$, $s\geq 3$ and $1 \leq i < s$ the spectrum of $L_i^\uparrow$ is
\[
    \begin{pmatrix} 0 & s+1 \\ (t+1)(st+1) \binom{s}{i} & (t+1)(st+1)\binom{s}{i+1} \end{pmatrix}. \qedhere
\]
If $i \geq s$, the spectrum of $L_i^\uparrow$ consists only of $0$ eigenvalues.
\end{lemma}

\begin{proof}
    Consider the clique complex $X$ of the point line graph $G$ of a $GQ(s,t)$. Then, no edge of $G$ (and thus, no $i$ simplex for $i > 0$ of $X$) is contained in more than one line of the $GQ$. Thus, in $X$, there are no $i$ simplex if $i > s$. Moreover, clearly $L_s^{\uparrow} = 0$.

    Now, fix $1 \leq i < s$. Then, the $(s+1)$ points on a line of the $GQ$ form $\binom{s+1}{i+1}$, $i-$ simplices. If $A$ and $B$ are two $i-$ simplices comprising of points from two distinct lines, then $|A \cap B| \leq 1$, and so, $L_i^{\uparrow}(A,B) = 0$. Thus, we see that $L_i^{\uparrow}$ has a block form with $(t+1)(st+1)$ blocks corresponding to the $(t+1)(st+1)$ lines in the $GQ.$ In addition, regardless of the ordering of the points on a line, the block is identical to the $i^{th}$ upper Laplacian matrix of the complex $K^{s}_{s+1}$. Thus, applying Proposition \ref{prop:kn-eigs} in each block and then taking the union of the spectra, we obtain the spectrum of $L_i^{\uparrow}$ as
    \[
    \begin{pmatrix} 0 & s+1 \\ (t+1)(st+1) \binom{s}{i} & (t+1)(st+1)\binom{s}{i+1} \end{pmatrix}. \qedhere
    \]
\end{proof}

\section{The spectra of the up-Laplacian \texorpdfstring{$L_1^\uparrow$}{L1-up} of some families of graphs} \label{sec:familes}

In this section, we determine the spectra of $L_1^\uparrow$ for some families of graphs.

\subsection{The Hamming Graphs}\label{subsec:Hamming}

Let $q$ and $n$ be two natural numbers. If $Q$ is a set of size $q$, the Hamming graph $H(q,n)$ has the collection $Q^n$ of ordered $n$-tuples or words of length $n$ with entries in $Q$ as its vertex set. Two words are adjacent if they differ in exactly one coordinate position.
The Hamming graph $H(q,n)$ is strongly regular when $q=2$ and the parameters of $H(n,2)$ are $(n^2,2(n-1),n-2,2)$. A convenient way to understand $H(2,n)$ is to view its vertices as the $n^2$ points of a $n\times n$ grid, where two points are adjacent if they are in the same row or in the same column.
\begin{theorem}\label{thm:hammingn2}
The spectrum of $L_1^\uparrow(H(2,n))$ is
\begin{equation*}
\begin{pmatrix}
0 & n\\
2n(n-1)&n(n-1)(n-2)
\end{pmatrix}.
\end{equation*}
\end{theorem}
\begin{proof}
By definition of $H(2,n)$, since the only triangles are on the same row or same column, $L_1^\uparrow(H(2,n))$ is a block diagonal matrix with each of the $2n$ blocks being $L_1^\uparrow(K_n)$. From Proposition \ref{prop:kn-eigs}, we know that
$L_1^\uparrow(K_n)$ has spectrum
$\begin{pmatrix}
0 & n\\
n-1 & \binom{n-1}{2}
\end{pmatrix}$. Combining these results, we obtain the desired result. 
\end{proof}

The argument above can be extended to the Hamming graph $H(q,n)$ for $q > 2$ by observing that the Laplacian $L_1^{\uparrow}(H(q,n))$ has $qn^{q-1}$ blocks equal to $L_1^{\uparrow}(K_n)$.
\begin{theorem}\label{thm:Hqn}
    The spectrum of $L_1^\uparrow(H(q,n))$ is
    \begin{equation*}
    \begin{pmatrix}
    0 & n\\
    qn^{q-1}(n-1) & qn^{q-1} \binom{n-1}{2}.
    \end{pmatrix}.
\end{equation*}
\end{theorem}

\subsection{The Triangular Graphs}\label{subsec:Tn}

For a natural number $n\geq 2$, the Triangular graph $T_n$ is the line graph of the complete graph $K_n$. Its vertices are the $2$-subsets of the set $\{1,\ldots,n\}$, where two $2$-subsets are adjacent if their intersection has size one. The Triangular graph $T_n$ is a strongly regular with parameters $(\binom{n}{2},2(n-2),n-2,4)$. In this section, we prove the following result. 
\begin{theorem}\label{thm:Tn}
For $n\geq 4$, the spectrum of $L_1^{\uparrow}(T_n)$ is
\begin{equation}
\begin{pmatrix}
0 & 2 & n-1 & n & n+2\\
\binom{n}{2}-1 & \binom{n-1}{2} & \frac{n(n-2)(n-4)}{3} & \binom{n-1}{2} & \binom{n-1}{3}
\end{pmatrix}.
\end{equation}
\end{theorem}

When there is no risk of confusion, we will use $L_1^{\uparrow}$ to denote $L_1^{\uparrow}(T_n)$, $G$ to denote $T_n$, and $E$ to denote $E(T_n)$ in the next subsections.

\subsubsection{Eigenvalue \texorpdfstring{$0$}{0} }

For $1\leq i<j\leq n$, define the vector $v_{i,j}\in \mathbb{R}^{E}$ as follows:
\begin{equation*}
v_{i,j}(e)=\begin{cases}
0, \text{if } \{i,j\}\notin e,\\
+1, \text{if } e=\{\{i,j\},\{r,s\}\} \text{ and } \{i,j\}<\{r,s\},\\
-1, \text{if } e=\{\{i,j\},\{r,s\}\} \text{ and } \{i,j\}>\{r,s\}.
\end{cases}
\end{equation*}

It is known \cite[Ch.4]{Biggs} that the vectors $v_{i,j}, 1\leq i<j\leq n$ sum up to the zero vector and span the cut space $\im \delta_0$ whose dimension is $\binom{n}{2}-1$. We have a self-contained proof showing that these vectors actually span $\ker \delta_1=\ker L_1^\uparrow$, but we omit this proof here for brevity and refer the reader for Subsection \ref{subsec:n_1} for an argument showing that the multiplicity of $0$ is $\binom{n}{2}-1$.

\subsubsection{Eigenvalue 2}

In this subsection, we show that $2$ is an eigenvalue of $L_1^{\uparrow}$ with multiplicity at least $\binom{n-1}{2}$.

As mentioned in Lemma \ref{nonzero}, $L_2^{\downarrow}=\delta_1\delta_1^{T}$ and $L_1^{\uparrow}=\delta_1^{T}\delta_1$ have the same positive eigenvalues (including their multiplicities). Note that $L_2^{\downarrow}$ is a matrix with constant row sums, where every diagonal entry is equal to $3$ and off-diagonal entries are in $\{-1,0,1\}$.

For any graph $X$ with an ordering of the vertices $V(X)$, we will let $A$ and $G$ be as follows: we write
\[
 L_2^{\downarrow} = 3I + A,
 \]
 where $A$ is a signed adjacency matrix for an underlying graph $G$. We say that $G$ with the signing given in $A$, is the \textsl{signed graph of $L_2^{\downarrow}(X)$}. 
 
 The vertices of $G$ are the triangles of $X$. A vertex of $G$ inherits an ordering on its three elements from the vertex ordering of $G$. For a triple $t =\{a,b,c\} \in V(X)$, we say that $t$ is of  \textsl{positive $\{a,b\}$-type} if $a$ and $b$ occur consecutively in the ordering of $t$ and we write that $\sgn_{a,b}(t) = 1$. We say that $t$ is of  \textsl{negative $\{a,b\}$-type} if $a$ and $b$ do not occur consecutively in the ordering of $t$; that is the ordering is either $a<c<b$ or $b<c<a$. In this case, we write that $\sgn_{a,b}(t) = -1$. We have the following lemma about the entries of $L_2^{\downarrow}(X)$. 
 
 \begin{lemma} Let $X$ be a graph. For any triangles $t_1$ and $t_2$ of $X$,  
  \[
 L_2^{\downarrow}(t_1, t_2) = \begin{cases} 3, & \text{if } t_1 = t_2; \\
  \sgn_{u,v}(t_1)\sgn_{u,v}(t_2), & \text{if } t_1 \cap t_2 = \{u,v\};\\
  0, & \text{otherwise.} \end{cases}
 \] 
 \end{lemma}
 
 \proof We observe that $\sgn_{a,b}(\{a,b,c\}) = [\{a,b,c\}: \{a,b\}]$. The result now follows from Lemma \ref{lem:epsLdown}. \qed 
 
\begin{lemma} The matrix $L_1^{\uparrow} = \delta_1^T\delta_1$ of $T_n$ has eigenvalue $2$ with multiplicity at least $\binom{n-1}{2}$. \end{lemma}

\proof We note that the $2$-eigenspace of $ L_2^{\downarrow}$ is equal to the $-1$-eigenspace of $A$. We will now construct $\binom{n-1}{2}$ linearly independent eigenvectors of $A$ with eigenvalue $-1$ by defining a vector $\Zv^{uv}$ for every unordered pair of vertices in $[n-1]$ and showing they are linearly independent. 
 
We consider the vertices of $K_n$ to be $[n] = \{1,\ldots,n\}$.  The vertices of $G$ are triples of edges of $K_n$ which induce either a triangle or a claw graph (also known as $K_{1,3}$ the complete bipartite graph with parts of size $1$ and $3$). Fix an edge $\{u,v\}$ of $K_n$ (a vertex of $T(n)$); we will write $uv$ for convenience when the meaning is clear. We will now construct $\Zv^{uv}$ for $G$, using the triples containing the edge $uv$. To simplify our notation, we will use $
\sgn(t) := \sgn_{u,v}(t)$ of the proof.

For two distinct vertices $u,v$ in $[n-1]$, we consider the following $n-2$ triangles:
\[
t_w = \{ uw, uv, vw \}, w\in [n] \setminus \{u,v\}.
\]
We consider also the following $(n-2)(n-3)$ claws of $K_n$
\[
c_{x}^w = \{ wu, wv,wx\}
\] 
for each $w \in [n] \setminus \{u,v\}$ and $x\in [n] \setminus \{u,v,x\}$. We note that these triples are the exactly vertices of $G$ which contain both $uw$ and $vw$. We observe that the subgraph of $G$ induced by $\{t_w\}$ and $\{ c_x^w\}$ is a disjoint union of $(n-2)$ copies of $K_{n-2}$; for each $w$, the vertices $\{t_w\} \cup \{c_x^w\}_{x}$ form a clique and have no other neighbours in the set specified. 

We define the vector $\Zv^{uv}$ as follows:
\[
\Zv^{uv}(t) = \begin{cases} -\sgn(t)(n-3), &\text{if } t = t_w \text{ for some }w; \\
\sgn(t), &\text{if } t = c_x^w \text{ for some }w,x; \\ 
0  &\text{otherwise}.
\end{cases} 
\]
Consider $t_w$ for $w \in [n] \setminus \{u,v\}$. We will write $\sgn(t) = \sgn_{uw,vw}(t)$. We have that 
\[
\begin{split} (A\Zv^{uv})(t_w) &= \sum_{\tau} A(t_w, \tau) \Zv^{uv}(\tau)  \\
&=\sum_{x  \in [n] \setminus \{u,v,w\}} A(t_w, c_x^w) \Zv^{uv}(c_x^w) \\
&= \sum_{x  \in [n] \setminus \{u,v,w\}} \sgn(t_w) \sgn(c_x^w) \sgn(c_x^w) \\
&=  \sgn(t_w)(n-3) \\
&= - \Zv^{uv}(t_w)\end{split} \] 
Similarly, for
 $c_w^x$  for $w \in [n] \setminus \{u,v\}$ and $x\in [n] \setminus \{u,v,x\}$. We have 
 \[
 \begin{split} (A\Zv^{uv})(c_x^w) &= \sum_{\tau} A(c_x^w, \tau) \Zv^{uv}(\tau)\\
 &= -(n-3) \sgn(t_w)^2 \sgn(c_x^w) + \sum_{y  \in [n] \setminus \{u,v,w,x\}} A(c_x^w, c_y^w) \Zv^{uv}(c_y^w) \\ 
 &= -(n-3) \sgn(c_x^w) + \sum_{y  \in [n] \setminus \{u,v,w,x\}} \sgn(c_x^w)\sgn(c_y^w) \sgn(c_y^w) \\
 &= -\sgn(c_x^w). \end{split}\]
Thus, we have shown that $A\Zv^{uv} = - \Zv^{uv}(t) $.

To see that $\{\Zv^{ab} \}_{a,b \in [n-1]}$ are linearly independent vectors, we observe that $\Zv^{uv}$ is the only vector with a non-zero entry corresponding to the edge set of triangle $u,v,n$ of $K_n$ and so cannot be a linear combination of the other vectors. 
 \qed 

\subsubsection{\texorpdfstring{Eigenvalue $n+2$}{n+2}}

In this subsection, we prove that $n+2$ is an eigenvalue of $L_1^{\uparrow}$ with multiplicity at least $\binom{n-1}{3}$.

We will need the following result.
\begin{prop}\label{prop:eig_6_T4}
The number $6$ is an eigenvalue of the matrix $L_1^{\uparrow}(T_4)$
\end{prop}
\begin{proof}
For the triangular graph $T_4$ in Figure \ref{fig:T4}
\begin{figure}[htbp]
\centering
\scalebox{0.8}{
\begin{tikzpicture}[scale=0.5]
\draw[fill=black] (-5,0) circle (3pt);
\draw[fill=black] (-8,0) circle (3pt);
\draw[fill=black] (-5,3) circle (3pt);
\draw[fill=black] (-8,3) circle (3pt);
\draw (-5,0) -- (-5,3) -- (-8,3)--(-8,0) --(-5,0)--(-8,3);
\draw (-5,3) -- (-8,0);
\draw (-8,3) node[above]{$u_1$};
\draw (-5,3) node[above]{$u_2$};
\draw (-5,0) node[below]{$u_3$};
\draw (-8,0) node[below]{$u_4$};

\node at (-3,0) {$v_6:u_3u_4$};
\node at (-3,1){$v_5:u_2u_4$};
\node at (-3,2){$v_4:u_2u_3$};
\node at (-3,3){$v_3:u_1u_4$};
\node at (-3,4){$v_2:u_1u_3$};
\node at (-3,5){$v_1:u_1u_2$};

\draw[fill=black] (0,0) circle (3pt);
\draw[fill=black] (3,0) circle (3pt);
\draw[fill=black] (6,0) circle (3pt);
\draw[fill=black] (1.5,2.6) circle (3pt);
\draw[fill=black] (3,5.2) circle (3pt);
\draw[fill=black] (4.5,2.6) circle (3pt);
\node at (0,-0.5) {$v_1$};
\node at (3,-0.5) {$v_4$};
\node at (6,-0.5) {$v_5$};
\node at (1,2.6) {$v_2$};
\node at (5,2.6) {$v_6$};
\node at (3,5.7) {$v_3$};
\draw[thick] (0,0)--(3,0)--(6,0)--(4.5,2.6)--(3,5.2)--(1.5,2.6)--(0,0);
\draw[thick] (1.5,2.6)--(3,0)--(4.5,2.6)--(1.5,2.6);
\draw[thick] (3,5.2) arc (90:210:3.46);
\draw[thick] (6,0) arc (-30:90:3.46);
\draw[thick] (0,0) arc (210:330:3.46);
\end{tikzpicture}}
\caption{The triangular graph $T_4$ as line graph of $K_4$}
\label{fig:T4}
\end{figure}
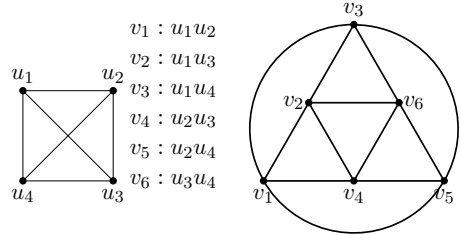
the boundary map $\partial_2=\delta_1^{T}$ is the $|E(T_4)|\times |T(T_4)|$ matrix in equation \eqref{eq:partial2T4} (where the rows $ij$ refer to the edges $v_iv_j$ of $T_4$ and the columns $ijk$ refer to the triangles $v_iv_jv_k$ in $T_4$), and $T(T_4)$ is the set of triangles in $T_4$):
{\small 
\begin{equation}\label{eq:partial2T4}
\delta_1^{T}=\partial_2=
\begin{blockarray}{ccccccccc}
&123 & 124 & 135 & 145 & 236 & 246 & 356 & 456 \\
\begin{block}{c[cccccccc]}
12&+1&+1&0&0&0&0&0&0\\
13&-1&0&+1&0&0&0&0&0 \\ 
14&0&-1&0&+1&0&0&0&0 \\ 
15&0&0&-1&-1 &0&0&0&0\\
23&+1&0&0&0&+1&0&0&0\\
24&0&+1&0&0&0&+1&0&0\\
26&0&0&0&0&-1&-1&0&0\\
35&0&0&+1&0&0&0&+1&0\\
36&0&0&0&0&+1&0&-1&0\\
45&0&0&0&+1&0&0&0&+1\\
46&0&0&0&0&0&+1&0&-1\\
56&0&0&0&0&0&0&+1&+1\\
\end{block}
\end{blockarray}\ .
\end{equation}}

The Laplacian $L_1^{\uparrow}(T_4)$ equals $\delta_1^T\delta_1$ and we write it below:
{\small 
\begin{equation}\label{eq:L1upT4}
\begin{blockarray}{ccccccccccccc}
&12 & 13 & 14 & 15 & 23 & 24 & 26 & 35 & 36 & 45 & 46 & 56 \\
\begin{block}{c[cccccccccccc]}
12 & 2 & -1 & -1 & 0 & +1 & +1& 0 & 0 & 0 & 0 & 0 & 0\\
13 & -1& 2  & 0  &-1 & -1  & 0 & 0  & +1& 0 & 0 & 0 & 0\\
14 & -1& 0  & 2  &-1 &  0  & -1& 0  & 0. & 0 & +1 & 0 & 0\\
15 & 0 & -1 & -1 & 2 & 0 & 0 & 0 & -1 & 0 & -1 & 0 & 0\\
23 & +1 & -1& 0 & 0 & 2 & 0 & -1 & 0 & +1 & 0 & 0 & 0\\
24 & +1& 0 & -1 & 0 & 0 & 2 & -1 & 0 & 0 & 0 & +1 & 0\\
26 & 0  & 0 & 0  & 0 & -1& -1& 2 & 0 & -1 & 0 & -1 & 0\\
35 & 0 & +1& 0 & -1& 0& 0  & 0 & 2  & -1 & 0 & 0 & +1\\
36 & 0 & 0  & 0 & 0 & +1& 0 & -1 & -1 & 2 & 0 & 0 & -1\\
45 & 0 & 0 & +1 & -1 & 0& 0 & 0 & 0 & 0 & 2 & -1 & +1\\
46 & 0 & 0 & 0 & 0 & 0& +1 & -1 & 0 & 0 & -1 & 2 & -1\\
56 & 0 & 0 & 0 & 0 & 0& 0 & 0 & +1 & -1 & +1 & -1 & 2\\
\end{block}
\end{blockarray}\ .
\end{equation} 
}

One can check directly by matrix multiplication that the following identity holds:
{\small 
\begin{equation}\label{eq:eigT4}
\begin{blockarray}{ccccccccccccc}
&12 & 13 & 14 & 15 & 23 & 24 & 26 & 35 & 36 & 45 & 46 & 56  \\
\begin{block}{c[cccccccccccc]}
12 & 2 & -1 & -1 & 0 & +1 & +1& 0 & 0 & 0 & 0 & 0 & 0\\
13 & -1& 2  & 0  &-1 & -1  & 0 & 0  & +1& 0 & 0 & 0 & 0\\
14 & -1& 0  & 2  &-1 &  0  & -1& 0  & 0. & 0 & +1 & 0 & 0\\
15 & 0 & -1 & -1 & 2 & 0 & 0 & 0 & -1 & 0 & -1 & 0 & 0\\
23 & +1 & -1& 0 & 0 & 2 & 0 & -1 & 0 & +1 & 0 & 0 & 0\\
24 & +1& 0 & -1 & 0 & 0 & 2 & -1 & 0 & 0 & 0 & +1 & 0\\
26 & 0  & 0 & 0  & 0 & -1& -1& 2 & 0 & -1 & 0 & -1 & 0\\
35 & 0 & +1& 0 & -1& 0& 0  & 0 & 2  & -1 & 0 & 0 & +1\\
36 & 0 & 0  & 0 & 0 & +1& 0 & -1 & -1 & 2 & 0 & 0 & -1\\
45 & 0 & 0 & +1 & -1 & 0& 0 & 0 & 0 & 0 & 2 & -1 & +1\\
46 & 0 & 0 & 0 & 0 & 0& +1 & -1 & 0 & 0 & -1 & 2 & -1\\
56 & 0 & 0 & 0 & 0 & 0& 0 & 0 & +1 & -1 & +1 & -1 & 2\\
\end{block}
\end{blockarray}
\begin{bmatrix}
+1\\
-1\\
-1\\
+1\\
+1\\
+1\\
-1\\
-1\\
+1\\
-1\\
+1\\
-1
\end{bmatrix}=6\cdot \begin{bmatrix}
+1\\
-1\\
-1\\
+1\\
+1\\
+1\\
-1\\
-1\\
+1\\
-1\\
+1\\
-1
\end{bmatrix}
\end{equation}}
This proves that $6$ is an eigenvalue of $L_1^{\uparrow}$.
\end{proof}
We now state and prove the main result of this subsection.

\begin{prop}
For $n\geq 4$, $n+2$ is an eigenvalue of $L_1^{\uparrow}(T_n)$
with multiplicity at least $\binom{n-1}{3}$.\end{prop}
\begin{proof}
For simplicity, let $E$ be the edge-set of $T_n$. For $1\leq a<b<c\leq n-1$, we define a vector $u_{a,b,c}\in \mathbb{R}^E$ as follows. We call the subgraph of $T_n$ in Figure \ref{fig:graph_abc} the graph associated with the vector $u_{a,b,c}$ and denote it by $G_{a,b,c}$. 

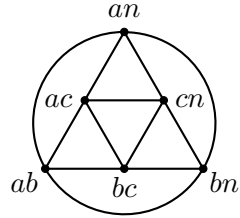
\begin{figure}[htbp]
\centering
\begin{tikzpicture}[scale=0.35]
\draw[fill=black] (0,0) circle (4pt);
\draw[fill=black] (3,0) circle (4pt);
\draw[fill=black] (6,0) circle (4pt);
\draw[fill=black] (1.5,2.6) circle (4pt);
\draw[fill=black] (3,5.2) circle (4pt);
\draw[fill=black] (4.5,2.6) circle (4pt);
\node at (-0.8,-0.5) {$ab$};
\node at (3,-0.7) {$bc$};
\node at (6.8,-0.5) {$bn$};
\node at (0.5,2.6) {$ac$};
\node at (5.5,2.6) {$cn$};
\node at (3,6) {$an$};
\draw[thick] (0,0)--(3,0)--(6,0)--(4.5,2.6)--(3,5.2)--(1.5,2.6)--(0,0);
\draw[thick] (1.5,2.6)--(3,0)--(4.5,2.6)--(1.5,2.6);
\draw[thick] (3,5.2) arc (90:210:3.46);
\draw[thick] (6,0) arc (-30:90:3.46);
\draw[thick] (0,0) arc (210:330:3.46);
\end{tikzpicture}
\caption{The graph of the vector $u_{a,b,c}$}
\label{fig:graph_abc}
\end{figure}
For each edge of $T_n$ that is not an edge of $G_{a,b,c}$, the entry of $u_{a,b,c}$ corresponding to that edge is $0$. For each edge in the graph $G_{a,b,c}$, the value of $u_{a,b,c}$ corresponding to that edge is $+1$ or $-1$ as described below:
\[
\begin{blockarray}{cc}
\begin{block}{c[c]}
ab,ac & +1 \\
ab,an & -1\\
ab,bc & -1\\
ab,bn & +1\\
ac,an & +1\\
ac,bc & +1\\
ac,cn & -1\\
an,bn & -1\\
an,cn & +1\\
bc,bn & -1\\
bc,cn & +1\\
bn,cn & -1\\
\end{block}
\end{blockarray}
\]
This vector is the same as the eigenvector of $L_1^{\uparrow}(T_4)$ appearing in equation \eqref{eq:eigT4} under the isomorphism between $T_4$ and $G_{a,b,c}$ given by $1\mapsto ab, 2\mapsto ac, 3\mapsto an, 4\mapsto bc, 5\mapsto bn, 6\mapsto cn$.

For arbitrary $n \geq 4$, each diagonal entry of $L_1^{\uparrow}(T_n)$ is $n-2$ because any edge in $T_n$ is contained in exactly $\lambda=n-2$ triangles. The principal submatrix of $L_1^{\uparrow}(T_n)$ corresponding to the edges in $G_{a,b,c}$ is the following:
{\small 
\begin{equation}\label{eq:eigT4a}
\begin{blockarray}{ccccccccccccc}
&12 & 13 & 14 & 15 & 23 & 24 & 26 & 35 & 36 & 45 & 46 & 56\\
\begin{block}{c[cccccccccccc]}
12 &n-2 & -1 & -1 & 0 & +1 & +1& 0 & 0 & 0 & 0 & 0 & 0\\
13 & -1&n-2  & 0  &-1 & -1  & 0 & 0  & +1& 0 & 0 & 0 & 0\\
14 & -1& 0  &n-2  &-1 &  0  & -1& 0  & 0. & 0 & +1 & 0 & 0\\
15 & 0 & -1 & -1 &n-2 & 0 & 0 & 0 & -1 & 0 & -1 & 0 & 0\\
23 & +1 & -1& 0 & 0 &n-2 & 0 & -1 & 0 & +1 & 0 & 0 & 0\\
24 & +1& 0 & -1 & 0 & 0 &n-2 & -1 & 0 & 0 & 0 & +1 & 0\\
26 & 0  & 0 & 0  & 0 & -1& -1&n-2 & 0 & -1 & 0 & -1 & 0\\
35 & 0 & +1& 0 & -1& 0& 0  & 0 &n-2  & -1 & 0 & 0 & +1\\
36 & 0 & 0  & 0 & 0 & +1& 0 & -1 & -1 &n-2 & 0 & 0 & -1\\
45 & 0 & 0 & +1 & -1 & 0& 0 & 0 & 0 & 0 &n-2 & -1 & +1\\
46 & 0 & 0 & 0 & 0 & 0& +1 & -1 & 0 & 0 & -1 &n-2 & -1\\
56 & 0 & 0 & 0 & 0 & 0& 0 & 0 & +1 & -1 & +1 & -1 &n-2\\
\end{block}
\end{blockarray}
\end{equation}}
The matrix above is the sum of the matrix in \eqref{eq:L1upT4} and $(n-4)I_{12}$. This fact and \eqref{eq:eigT4} imply that $u_{a,b,c}$ is an eigenvector of $L_1^{\uparrow}(T_n)$ corresponding to the eigenvalue $6+n-4=n+2$.


The eigenvectors $u_{a,b,c}$, for $1\leq a<b<c\leq n$, are linearly independent. This is because given any entry $(ij,i\ell)$ say, with $i<j<\ell$, there is exactly one of the vectors $u_{a,b,c}$ that has a non-zero entry in the position $(ij,i\ell)$, namely the vector $u_{i,j,\ell}$. These facts imply that $n+2$ is an eigenvalue of $L_1^{\uparrow}(T_n)$ with multiplicity at least $\binom{n-1}{3}$. \end{proof}

\subsubsection{Eigenvalue \texorpdfstring{$n$}{n}}\label{eig_n}

The main result of this subsection is the following statement.
\begin{prop}\label{prop:eig_n}
The number $n$ is an eigenvalue of $L_1^{\uparrow}$ with multiplicity at least $\binom{n-1}{2}$.
\end{prop}
\begin{proof}

For $1 \leq a < b \leq n-1$, we will define a vector $w_{a,b}\in \re^{X_1}=\re^{E}$ and we will prove that $w_{a,b}, 1\leq a<b\leq n-1$ are linearly independent eigenvectors of $L_1^{\uparrow}$ corresponding to the eigenvalue $n$. 

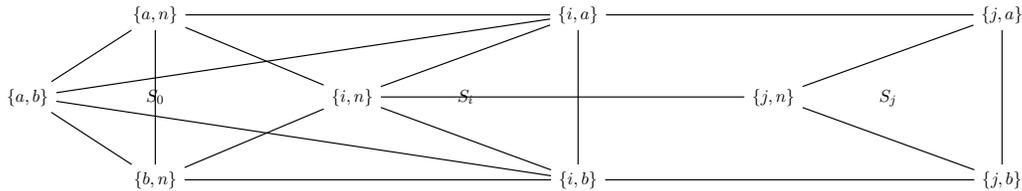
\begin{figure}[htbp]
    \centering
    \[\begin{tikzcd}[scale cd=0.6]
	& {\{a,n\}} &&&& {\{i,a\}} &&&& {\{j,a\}} \\
	{\{a,b\}} & {S_0} && {\{i,n\}} & {S_i} &&& {\{j,n\}} & {S_j} \\
	& {\{b,n\}} &&&& {\{i,b\}} &&&& {\{j,b\}}
	\arrow[no head, from=1-2, to=1-6]
	\arrow[no head, from=1-2, to=2-1]
	\arrow[no head, from=1-2, to=2-4]
	\arrow[no head, from=1-2, to=3-2]
	\arrow[no head, from=1-6, to=1-10]
	\arrow[no head, from=1-6, to=3-6]
	\arrow[no head, from=1-10, to=3-10]
	\arrow[no head, from=2-1, to=1-6]
	\arrow[no head, from=2-1, to=3-2]
	\arrow[no head, from=2-1, to=3-6]
	\arrow[no head, from=2-4, to=1-6]
	\arrow[no head, from=2-4, to=2-8]
	\arrow[no head, from=2-4, to=3-6]
	\arrow[no head, from=2-8, to=1-10]
	\arrow[no head, from=2-8, to=3-10]
	\arrow[no head, from=3-2, to=2-4]
	\arrow[no head, from=3-2, to=3-6]
	\arrow[no head, from=3-6, to=3-10]
    \label{eigenvalue_n_proof}
    \end{tikzcd}\]
    \caption{The triangles $S_0,\  S_i,$ and $S_j$ for some distinct $i,j\notin\{a,b,n\}$}
    \label{fig:eigenvalue_n_proof}
\end{figure}

Denote $S_0 = \{\{a,b\}, \{a,n\}, \{b,n\}\}$. For $1 \leq i \leq n-1, i \not \in \{a,b\}$, let $\ S_i = \{\{i,a\},\{i,b\},\{i,n\}\}$. For any triangle $\mathcal{T}$, denote by $v_{\mathcal{T}}$ the indicator vector of $\mathcal{T},$ i.e., the vector indexed by the triangles in $G$ and whose only nonzero entry is $1$ at the position indexed by $\mathcal{T}$. 

The triangles in the set $\{S_i: 0\leq i\leq n-1, i\notin \{a,b\}\}$ are pairwise vertex disjoint and therefore, are pairwise edge disjoint (see Figure \ref{fig:eigenvalue_n_proof} for an illustration). Therefore, if $e$ and $f$ are two edges such that $e \in S_i$ and $f \in S_j$ for some $1\leq i \not = j\leq n-1, i,j\notin \{a,b\}$, then $L_1^{\uparrow}(e,f) = 0$ since no triangles in $G$ contains both $e$ and $f$. Hence, the vectors $\delta_1^T v_{S_i}, 0 \leq i \leq n-1, i \notin \{a,b\}$ (which are the columns of $\delta_1^T$ corresponding to the triangles $S_i, 0\leq i \leq n-1, i\notin \{a,b\}$) have disjoint supports. 

For $1\leq a<b<n$, define
\begin{equation*}
    w_{a,b} = \sum_{i=0: i\notin \{a,b\}}^{n-1} \delta_1^T v_{S_i}.
\end{equation*}
For simplicity, we denote $w=w_{a,b}$ for the rest of this proof. If $f\in E$, then
\begin{equation}\label{eq:w_f}
    w_f=\sum_{i=0:i\notin \{a,b\}}^{n-1}[S_i:f].
\end{equation}
This shows that the nonzero entries of $w_{a,b}$ are all $\pm 1'$s and are precisely at the positions corresponding to the edges of the triangles $\{S_i:0\leq i\leq n-1, i\notin \{a,b\}\}$.

We will now show that
\begin{equation}\label{eq:verA}
    L_1^{\uparrow}w = n w \text{ or }  (L_1^{\uparrow}w)_e = n w_e, \forall e\in E. 
\end{equation}

Let $e$ be an edge of $T_n$. Using \eqref{eq:w_f}, we obtain that
\begin{equation}\label{eq:verB}
    (L_1^{\uparrow}w)_e = \sum_{f \in E} L_1^{\uparrow}(e,f) w_f = \sum_{f \in E} L_1^{\uparrow}(e,f)\sum_{i=0:i\notin \{a,b\}}^{n-1}[S_i:f]   
    = \sum_{i=0:i\notin\{a,b\}}^{n-1}\sum_{f \in S_i} L_1^{\uparrow}(e,f) \ [S_i:f].
\end{equation}

If $e$ is not contained in a triangle with any edge of any $S_i$, then 
\begin{equation}\label{eq:e_no_triangle}
L_1^{\uparrow}(e,f)=0, \forall f\in S_i, 0\leq i\leq n-1, i\notin \{a,b\}. 
\end{equation}
In this case, $e$ is not contained in any $S_i$ and therefore, $w_e=0$. Hence, $(L_1^{\uparrow}w)_e=0=nw_e$.

We will group the remaining edges into several types and prove \eqref{eq:verA} each group separately. These categories are as follows: 
\begin{itemize}
    \item \textbf{Type 1:} $e \in S_{t}$ for some $t\in \{0,1,\ldots,n-1\}\setminus \{a,b\}$.
    \begin{proof}
        Assume that $e=\{x,y\} \subset S_{t}=\{x,y,z\}$, where $x < y.$ Using \eqref{eq:verB}, we get that
        \begin{align*}
            (L_1^{\uparrow}w)_e &= \sum_{f \in S_{t}} L_1^{\uparrow}(e,f) [S_{t}:f]\\
            &=L_1^{\uparrow}(e,e)[S_{t}:e]+L_1^{\uparrow}(e,\{y,z\})[S_{t}:\{y,z\}]+L_1^{\uparrow}(e,\{x,z\})[S_{t}:\{x,z\}] \\           
            &=\begin{cases}
                (n-2)(-1) + (-1)(+1) + (-1)(+1) = -n, & \text{ if } x < z < y\\
                (n-2)(1) + (1)(1) + (-1)(-1) = n, & \text{ otherwise} 
            \end{cases}\\
            &= \begin{cases}
                n[S_{t}:e], & \text{ if } x < z < y\\
                n[S_{t}:e], & \text{ otherwise} 
            \end{cases}\\
            &= nw_e.
        \end{align*}
    \end{proof}
    
    \item \textbf{Type 2:} The two endpoints of $e$ belong to $S_i$ and $S_j$ for some $i \not = j.$
    \begin{proof}
        If $i \not = 0$ and $j \not = 0,$ then $e$ must be one of the three edges 
        \begin{equation*}
            \{\{i,a\}, \{j,a\}\}, \{\{i,b\}, \{j,b\}\}, \{\{i,n\}, \{j,n\}\}. 
         \end{equation*}
        Since none of these three edges share a triangle with any of the edges in $\{S_m\}, m\in \{0,\ldots,n-1\}\setminus \{a,b\}$, using the same argument as the one preceding equation \eqref{eq:e_no_triangle}, we conclude that for these values of $e,$ both sides of \eqref{eq:verA} are $0$. 
        
        We may now assume $j = 0 < i$. There are precisely six edges that are incident with both a vertex in $S_0$ and a vertex in $S_i$ that we need to consider:
        \begin{gather}
            \{\{a,i\},\{a,n\}\}, \{\{b,i\},\{b,n\}\}\\
            \{\{a,b\},\{a,i\}\}, \{\{a,b\},\{b,i\}\},\\
            \{\{a,n\},\{i,n\}\}, \{\{b,n\},\{i,n\}\},\\
       \end{gather}
        We include two of these cases below, and the remaining four in Appendix \ref{appendix:e_n_type2}.
        \begin{itemize}
            \item $e = \{\{a,i\},\{a,n\}\}$. Using \eqref{eq:verB}, we obtain that
            \begin{gather}
                (L_1^{\uparrow}w)_e = \sum_{p=0:p\not \in\{a,b\}}^{n-1}\sum_{f \in S_p}L_1^{\uparrow}(e,f)[S_p:f]\\
                =L_1^{\uparrow}(e,\{\{a,b\},\{a,n\}\}) \cdot [S_0:\{\{a,b\},\{a,n\}\}] + L_1^{\uparrow}(e,\{\{i,a\},\{i,n\}\})\cdot [S_i:\{\{i,a\},\{i,n\}\}]\\
                = L_1^{\uparrow}(e,\{\{a,b\},\{a,n\}\}) - L_1^{\uparrow}(e,\{\{i,a\},\{i,n\}\})\\
                =-[e:\{a,n\}][\{\{a,b\},\{a,n\}\}:\{a,n\}]+[e:\{a,i\}][\{\{a,i\},\{i,n\}\}:\{a,i\}]\\
                =-(-1)(-1)+1\cdot 1=0=nw_e.
            \end{gather}

            \item $e=\{\{a,b\}, \{a,i\}\}$. Using \eqref{eq:verB}, we get that
            \begin{gather}
                (L_1^{\uparrow}w)_e = \sum_{p=0:p \not \in \{a,b\}}^{n-1} \sum_{f \in S_p}L_1^{\uparrow}(e,f) [S_p:f] \\
                = L_1^{\uparrow}(e, \{\{a,b\},\{a,n\}\})\ [S_0:\{\{a,b\},\{a,n\}\}] + L_1^{\uparrow}(e,\{\{i,a\},\{i,b\}\})\ [S_i:\{\{i,a\},\{i,b\}\}]\\
                = L_1^{\uparrow}(e, \{\{a,b\},\{a,n\}\}) + L_1^{\uparrow}(e,\{\{i,a\},\{i,b\}\})\\
                =-[e:\{a,b\}][\{\{a,b\},\{a,n\}\}:\{a,b\}]-[e:\{i,a\}][\{\{i,a\},\{i,b\}\}:\{i,a\}]\\
                =-[e:\{a,b\}]-[e:\{i,a\}]=0=nw_e.
            \end{gather}

        \end{itemize}

    \end{proof}
    \item \textbf{Type 3:} Only one endpoint of $e$ belongs to some $S_i$.
    \begin{proof}
        First, we show that $i \not = 0$. Indeed, if $e$ is an edge that is not in $S_0$ but shares a triangle with an edge in $S_0$, then the endpoint of $e$ not in $S_i$ must be of the form $\{a,x\}, \{b,x\},$ or $\{n,x\}.$ This is not possible since each of these vertices belong to $S_x$, contradicting the assumption that $e$ has only one of its endpoints in one of the $S_i'$s. Thus, we may assume $i > 0$. In turn, the endpoint of $e$ not in $S_i$ must be of the form $\{i,x\}$ for some $x \not \in \{a,b,n\}.$ Thus, there are three possible values of $e$ to consider: $\{\{i,a\},\{i,x\}\},\{\{i,b\},\{i,x\}\},$ and $ \{\{i,n\},\{i,x\}\}$.
        
        We verify \eqref{eq:verA} for one of these values below and we do the remaining two in Appendix \ref{appendix:e_n_type3}. 
        \begin{itemize}
            \item \textbf{Case 1:} $e = \{\{i,a\},\{i,x\}\}$. 
            Observe that if $f \in S_j$ for some $j \not =i$, we have that $L_1^{\uparrow}(e,f) = 0$ by the hypothesis on $e$. Using \eqref{eq:verB}, we get that
            \begin{gather*}
                (L_1^{\uparrow}w)(e)=\sum_{p=0:p \not \in \{a,b\}}^{n-1} \sum_{f \in S_p}L_1^{\uparrow}(e,f) [S_p:f] \\
                = \sum_{f \in S_i} L_1^{\uparrow}(e,f) [S_i:f]\\
                =L_1^{\uparrow}(e,\{\{i,a\},\{i,n\}\})[S_i:\{\{i,a\},\{i,n\}\}]+L_1^{\uparrow}(e,\{\{i,a\},\{i,b\}\})[S_i:\{\{i,a\},\{i,b\}\}]\\
                =-L_1^{\uparrow}(e,\{\{i,a\},\{i,n\}\})+L_1^{\uparrow}(e,\{\{i,a\},\{i,b\}\})\\
                =[e:\{i,a\}][\{\{i,a\},\{i,n\}\}:\{i,a\}]-[e:\{i,a\}][\{\{i,a\},\{i,b\}\}:\{i,a\}]\\
                =[e:\{i,a\}]-[e:\{i,a\}]=0=nw_{e}.
            \end{gather*}
            
        \end{itemize}
    \end{proof}
\end{itemize}

The vectors $w_{a,b}, 1\leq a<b\leq n-1$ are linearly independent. To see this, note that a vector $w_{c,d}$ has a non-zero entry at the coordinate position indexed by the edge $\{\{a,n\}, \{b,n\}\}$ if and only if $(a,b)=(c,d).$ Hence, the $n-$eigenspace of $L$ has dimension at least $\binom{n-1}{2}$.

\end{proof}

\subsubsection{Eigenvalue \texorpdfstring{$(n-1)$}{n-1}}\label{subsec:n_1}

Let us denote $L := L_1^{\uparrow}(T_n)$. Note that $T_n$ has $(n-2)\binom{n}{2}$ edges and therefore, $L$ is a square matrix of size $(n-2)\binom{n}{2}$. In addition, each diagonal entry of $L$ is $n-2$, since each edge is contained in exactly $n-2$ triangles. Thus,
\begin{equation}\label{eq:trL}
    \text{tr}(L) = (n-2)^2 \binom{n}{2} = \frac{n(n-1)(n-2)^2}{2}.
\end{equation}
Except for the diagonal entries, each row of $L$ has $2(n-2)$ nonzero entries, each $\pm 1$, which correspond to the other sides of the $n-2$ triangles containing this particular edge. Hence, the inner product of each row with itself equals $(n-2)^2 + 2(n-2)$. Therefore,
\begin{equation}\label{eq:trL2}
    \text{tr}(L^2) = (n-2)\binom{n}{2} ((n-2)^2+2(n-2)) = n(n-2)^2\binom{n}{2} = \frac{n^2(n-1)(n-2)^2}{2}.
\end{equation}
At this point, we have proved that $0$ has multiplicity at least $\binom{n}{2}-1$, $2$ has multiplicity at least $\binom{n-1}{2}$, $n$ has multiplicity at least $\binom{n-1}{2}$ and $n+2$ has multiplicity at least $\binom{n-1}{3}$. Denote by $\mu_1, \dots, \mu_t$ the remaining eigenvalues. We have that
\begin{equation*}
    t = (n-2) \binom{n}{2} - \Big(\binom{n}{2}-1\Big)-\binom{n-1}{2}-\binom{n-1}{2}-\binom{n-1}{3} = \frac{n(n-2)(n-4)}{2}.
\end{equation*}
Using \eqref{eq:trL}, we get that
\begin{gather*}
    \sum_{i=1}^{t} \mu_i
    = tr(L)  - 0 \cdot \Big(\binom{n}{2}-1 \Big)- 2 \cdot \binom{n-1}{2}-n \cdot \binom{n-1}{2}-(n+2) \cdot \binom{n-1}{3} \\
    = \frac{n(n-1)(n-2)(n-4)}{3},
\end{gather*}
From \eqref{eq:trL2}, we deduce that
\begin{gather*}
    \sum_{i=1}^{t} \mu_i^2 = tr(L^2) - 0^2 \cdot \Big(\binom{n}{2}-1 \Big)- 2^2\cdot \binom{n-1}{2}-n^2 \cdot \binom{n-1}{2}-(n+2)^2 \cdot \binom{n-1}{3}\\
    = \frac{n(n-1)^2(n-2)(n-4)}{3}.
\end{gather*}
Finally, we check that
\begin{gather*}
    \cfrac{(\sum_{i=1}^t \mu_i)^2}{\sum_{i=1}^t\mu_i^2} = \cfrac{n^2(n-1)^2(n-2)^2(n-4)^2}{9} \cdot \cfrac{3}{n(n-1)^2(n-2)(n-4)}  \\
    = \cfrac{n(n-2)(n-4)}{3} = t = \sum_{i=1}^t 1^2,
\end{gather*}
that is, the nonnegative values $\mu_1, \dots, \mu_t$ satisfy the equality condition of the Cauchy- Schwarz inequality. Thus, we must have $\mu_1 = \dots = \mu_t$. Plugging $\mu_i = \mu_1$ for each $i > 1$ in the equation for $\sum_{j=1}^t \mu_j$ above, we obtain $\mu_1 = \dots = \mu_t = n-1$. That is, all remaining eigenvalues are equal to $n-1$, and the eigenvalue $n-1$ has multiplicity $t = \frac{n(n-2)(n-4)}{2}$.

\subsection{Higher Dimensional Up-Laplacian eigenvalues}\label{sec:higher-lap}

In this section, we compute the up- and down-Laplacian spectra at each level for the clique complex of the triangular graph $T_n$.

\subsubsection{The Spectrum of  \texorpdfstring{$L_2^{\uparrow}(T_n)$}{L2-Tn}}

\begin{prop} The spectrum of $L_2^{\uparrow}(T_n)$ is
\begin{gather*}
    \begin{pmatrix}
        0 & n-1\\
        \binom{n}{3} + n\binom{n-2}{2} & n\binom{n-2}{3}
    \end{pmatrix}.
\end{gather*}

\end{prop}
\proof

Let $L:=L_2^{\uparrow}(T_n)$. Let $U_1$ and $U_2$ be the collections of triangles of $T_n$ of the forms $\{ai,bi,ci\}$ and $\{ab,bc,ca\}$, respectively. If $t \in U_2$, then the row and the column indexed by $t$ in $L$ is $0$. Let $L'$ be the principal submatrix of $L$ obtained by removing the rows and columns indexed by elements of $U_2$.  The nonzero part of the spectrum of $L$ is the same as that of $L'$, and for the $0-$ eigenvalue, it is an additional $\binom{n}{2}$ (note that it is possible that $0$ is not an eigenvalue of $L'$). 

Now, compute the spectrum of $L'$.  If $t_1,t_2$ are two triangles in $U_1$, then they have the forms respectively $\{ai,bi,ci\}$ and $\{a'j,b'j,c'j\}$, for some $a,b,c,a',b',c',i,j$. Note that unless $i = j$, $L'(t_1,t_2)=0$. Thus, We can view the matrix $L'$ in the block form with blocks $B_1, \dots, B_n$, where in $B_i$, contains the rows and columns  indexed by the triangles $\{(ai,bi,ci): a,b,c \in [n] \setminus \{i\}, a,b,c \text{ pairwise distinct}\}$. Each block has dimension $\binom{n-1}{2}$. 

Now, for a fixed $i$, the block $B_i$ is in fact, the matrix for $L_2^{\uparrow}(K_{n-1}) = L_2^{\uparrow}(K^{3}_{n-1})$.  Note here that the original orientations of the simplices get carried over. Thus, the spectrum of $B_i$ is given by the lemma above.

\begin{gather*}
    \begin{pmatrix}
        0 & n-1\\
        \binom{n-2}{2} & \binom{n-2}{3}
    \end{pmatrix},
\end{gather*}
which is independent of $i$. Now, since there are $n$ such blocks in $L'$, the spectrum of $L'$ is obtained by taking $n$ copies of the spectrum above:
\begin{gather*}
    \begin{pmatrix}
        0 & n-1\\
        n\binom{n-2}{2} & n\binom{n-2}{3}
    \end{pmatrix}.
\end{gather*}
Finally, the spectrum of $L$ is obtained by including $\binom{n}{3}$ copies of $0$:
\begin{gather*}
    \begin{pmatrix}
        0 & n-1\\
        \binom{n}{3} + n\binom{n-2}{2} & n\binom{n-2}{3}
    \end{pmatrix}.
\end{gather*}

\subsubsection{The Spectrum of \texorpdfstring{$L_k^{\uparrow}(T_n)$ for $3 \leq k \leq n-3$}{Lk-Tn}}
\begin{prop}
    The spectrum of $L_k^{\uparrow}(T_n)$ for $3 \leq k \leq n-3$ is
    \begin{gather*}
    \begin{pmatrix}
        0 & n-1\\
        n\binom{n-2}{k} & n\binom{n-2}{k+1},
    \end{pmatrix}.
    \end{gather*}
\end{prop}
\proof Let $L := L^{\uparrow}_{k}(T_n)$. For $k \geq 3$, the cliques of order $k+1$ in $T_n$ are of the form $\{ia_1,\dots, ia_{k+1}\}$ for some pairwise distinct collection $i, a_1, \dots, a_{k+1}\in [n]$. Denote the block $B_i = \{(ia_1,\dots,ia_{k+1}): a_1, \dots, a_k \in [n] \setminus \{i\}, a_i \not = a_j \text{ if } i \not = j\}$. We observe that if $q \in B_i, q' \in B_j$ and $i \not = j$, then $L(q,q') = 0$. Hence, we can view the matrix $L_k$ as blocks indexed by $B_i$.

Let us compute the spectrum of $B_i$ for a fixed $i$. As in the previous proof, we observe that $B_i$ is matrix $L_k^{\uparrow}(K_{n-1}) = L_k^{\uparrow}(K_{n-1}^{k+1})$, which has spectrum
\begin{gather*}
    \begin{pmatrix}
        0 & n-1\\
        \binom{n-2}{k} & \binom{n-2}{k+1},
    \end{pmatrix}.
\end{gather*}
which is independent of $i$. Thus, the spectrum of $L$ is obtained by taking $n$ copies (since there are $n$ blocks $B_i$) of this spectrum:
\begin{gather*}
    \begin{pmatrix}
        0 & n-1\\
        n\binom{n-2}{k} & n\binom{n-2}{k+1},
    \end{pmatrix}.
\end{gather*}
as desired. 

\subsection{The Spectrum of \texorpdfstring{$L_k^{\downarrow}(T_n)$}{Lk-Tn}}
\begin{prop}
The spectrum of $L_k^{\downarrow}(T_n)$ for $k \geq 4$ is given by
    \begin{gather*}
    \begin{pmatrix}
        0 & n-1\\
        n\Big[\binom{n-1}{k}-\binom{n-2}{k}\Big] & n\binom{n-2}{k}.
    \end{pmatrix}.
\end{gather*}
\end{prop}
\proof
The proposition follows from the observation that the nonzero part of the spectrum of $L_k^{\downarrow}$ is the same as that of $L_{k-1}^{\uparrow}.$ The multiplicity of the $0-$ eigenvalue is obtained using the rank-nullity theorem and observing that the dimension of the domain of $L_{k-1}$ is $n \binom{n-1}{k}$.

\section{Graphs with trivial first cohomology groups}\label{subsec:Tch}

In the previous sections, we showed that the clique complexes of several classes of strongly regular graphs have the property that $\im \delta_0 = \ker \delta_1,$ or, equivalently, the first cohomology group $H^1$ is trivial. This is the same as $\dim \ker L_1^{\uparrow} =|V(G)|-1$ if $G$ is connected. 

In this section, we will develop some general techniques for proving the existence of this property in graphs. We apply these techniques to show that chordal graphs, certain conference graphs, and all but finitely many Paley graphs have trivial first cohomology groups. We also apply our findings to general strongly regular graphs and obtain a linear inequality involving the $k,\lambda,$ and $\mu$ parameters that guarantees a trivial first cohomology group. We start by introducing some notations and terminologies.

\begin{defn}
If $X_i = \{f_1,\dots,f_p\},$ then $\re^{X_i}$ is an $\re-$vector space with a basis $\{1_{f_j}\}_{j=1}^p,$ where $1_{f_j}$ is the indicator function of $f_j$ (which is the function taking value $1$ at $f_j,$ and $0$ everywhere else). Any given $g \in \re^{X_i}$ may be written uniquely in the form $g = \sum_{j=1}^p c_j 1_{f_j}, c_j \in \re.$ We will slightly abuse the notation and write $g = \sum_{j=1}^p c_j f_j,$ writing $f_j$ for $1_{f_j}.$ Further, if one or more of the $c_i$ are $0'$s, we may omit some or all of those terms in such an expression of $g.$ We also define the coefficient function $[f_j,g] := c_j$ for $f_j \in X_i.$
\end{defn}

\begin{defn}
    For some $g \in \re^{X_i},$ let $g = \sum_{j=1}^s c_j' f_j'$ be the unique representation of $g$ as described above where each $c_j' \not = 0.$ We call this the canonical representation of $g$ and we define the support of $g$ to be $\{f_j'\}_{j=1}^s$. 
\end{defn}


\begin{defn}
If $e$ is an edge and $v$ a vertex, we simplify the notation $[e:\{v\}]$ from \eqref{eq:signing} as $[e:v]$.
\end{defn}

We will now obtain a simple spanning set for $\ker \delta_0^T.$

\begin{defn}\label{def:TC}
Let $C=(v_1,\dots,v_{\ell})$ be an ordered tuple of vertices in $G$ such that $v_1,\dots,v_{\ell}$ form a cycle in $G$. Denoting $v_{\ell+1} = v_1$ and $e_i = \{v_i, v_{i+1}\}$ for $1 \leq i \leq \ell$, define
    \begin{gather}
        T(C) =  \sum_{i=1}^{\ell} c_i e_i \in \re^{X_1},
    \end{gather}
    where the coefficients $c_1,\dots, c_{\ell}$ are defined recursively as follows:
    \begin{gather}\label{eq:coeffTC}
        c_i = \begin{cases}
            1, & \text{ if } i=1\\
            -c_{i-1}[e_{i-1} : v_i] \cdot [e_i:v_i], & \text{ if } 2\leq i \leq \ell.
        \end{cases}
\end{gather}
For simplicity, we will use $T(v_1,\ldots,v_{\ell})$ to denote $T(C)$ when necessary.
\end{defn}

For the graph in Figure \ref{graph1}, we have that
\begin{gather}
    T(1,2,3) = \{1,2\}+\{2,3\}-\{1,3\}, \\
    T(1,3,2) = \{1,3\} - \{2,3\} - \{1,2\}, \\
    T(1,2,4,3) = \{1,2\} + \{2,4\} - \{3,4\} - \{1,3\}.
\end{gather}

\begin{lemma}\label{lem:TCkerdelta0}
If $C=(v_1,\dots,v_{\ell})$ is an ordered tuple of vertices in $G$ such that $v_1,\dots,v_{\ell}$ form a cycle, then $T(C) \in \ker \delta_0^T.$
\end{lemma}

\begin{proof}
Let $v$ be a vertex of $G$. The non-zero entries of $T(C)$ correspond to the edges of $C$ and the non-zero entries of the $v$-th row of $\delta_0^T$ are exactly those entries labeled by the edges containing $v$. Therefore, if $v\notin \{v_1,\ldots,v_{\ell-1},v_{\ell}\}$, then the coefficient of $v$ in $\delta_0^{T}T(C)$ is $0$.

It remains to check what happens when $v\in \{v_1,\ldots,v_{\ell-1},v_{\ell}\}$.

For $2 \leq i \leq \ell$, using \eqref{eq:coeffTC}, the coefficient of $v_i$ in $\delta_0^T\ T(C)$ is
    \begin{gather}
        c_{i-1}[e_{i-1}:v_{i}] + c_i[e_i:v_i] = c_{i-1}[e_{i-1}:v_{i}] -c_{i-1}[e_{i-1} : v_i] \cdot [e_i:v_i] \cdot [e_i:v_i] = 0.
    \end{gather}

%

We now show that the coefficient of $v_1$ is also $0.$ From \eqref{eq:coeffTC}, we have that $c_1,\ldots, c_{\ell}\in \{-1,1\}$ and that 
\begin{align*}
    c_2&=-c_1[e_1:v_2][e_2:v_2],\\
    c_3&=-c_2[e_2:v_3][e_3:v_3],\\
    \vdots &= \vdots \ \ \ \ \ \ \ \ \ \ \ \  \vdots \ \ \ \ \ \ \ \  \ \vdots \\
    c_{\ell}&=-c_{\ell-1}[e_{\ell-1}:v_{\ell}][e_{\ell}:v_{\ell}].
\end{align*}
Multiplying these equations and dividing both sides by $c_2\cdot c_3\cdot \ldots \cdot c_{\ell-1}\neq 0$, we get that 
\begin{align*}
    c_{\ell}&=(-1)^{\ell-1}\cdot c_1\cdot \prod_{i=2}^{\ell}[e_{i-1}:v_{i}][e_{i}:v_{i}]=(-1)^{\ell-1}\cdot 1\cdot [e_1:v_2][e_{\ell}:v_{\ell}]\prod_{j=2}^{\ell-1}([e_j:v_j][e_j:v_{j+1}])\\
    &=(-1)^{\ell-1}[e_1:v_2][e_{\ell}:v_{\ell}](-1)^{\ell-2}=-[e_1:v_2][e_{\ell}:v_{\ell}].
\end{align*}
Therefore,
\begin{equation*}
    [e_1:v_1]c_1+[e_{\ell}:v_1]c_{\ell}=[e_1:v_1]-[e_{\ell}:v_1][e_{\ell}:v_{\ell}][e_1:v_2]=[e_1:v_1]+[e_1:v_2]=0.
\end{equation*}

\end{proof}

The vectors of the form $T(C)$ form a spanning set for $\ker \delta_0^T$. 
\begin{prop}\label{prop:spanTC}
The elements of the set 
\begin{equation}
\{T(C): C=(v_1,\dots,v_{\ell}) \text{ is an ordered tuple of vertices such that } v_1,\dots,v_{\ell} \text{ is a cycle}\}
\end{equation}
span $\ker \delta_0^T.$
\end{prop}

This result follows from Biggs \cite[Ch.4,5]{Biggs} and Proposition \ref{prop:TC} below. Biggs uses a slightly different notation which we briefly recall now. Let $\Gamma$ be a graph with $n$ vertices and $m$ edges. For each edge $e_{j}=\{v, w\}$ of $\Gamma$, we may define an orientation on the edge $e_{j}$ by assigning one of the vertices $\{v,w\}$ as the positive end and the other vertex as the negative end. With an orientation chosen for each edge of $\Gamma$, the incidence matrix $D=(d_{ij})_{1\leq i\leq n, 1\leq j\leq m}$ is the $n \times m$ matrix whose entries are
$$
d_{ij} = \begin{cases}
+1, & \text{if } v_i \text{ is the positive end of } e_j,\\
-1, & \text{if } v_i \text{ is the negative end of } e_j,\\
0, & \text{otherwise.}\\
\end{cases} 
$$
Note that $\delta_0^T$ is an incidence matrix. 

The cycle-subspace of a connected $\Gamma$ is defined as the kernel of the linear transformation whose matrix is $D$. If $Q$ is a set of edges in $\Gamma$ that form a cycle in $\Gamma$, we may choose one of the two cyclic orderings/orientations of the edges in $Q$, and define a vector $\xi_Q \in \re^{X_1}$ by 
\begin{equation*}
\xi_Q(e_j)=
\begin{cases}
  +1, \text{ if } e_j \in Q \text{ and the orientation of } e_j \text{ in } Q \text{ matches its orientation in } \Gamma, \\
  -1, \text{ if } e_j\in Q \text{ and the orientation of } e_j \text{ in } Q  \text{ is not the same as its orientation in } \Gamma,\\ 
  0, \text{ if } e_j \not \in Q.     
\end{cases}
\end{equation*}
 It is known that $\xi_Q$ is in the cycle space of $\Gamma$ (see \cite[Prop 4.5]{Biggs}). If $T$ is a spanning tree of $\Gamma$, then for each edge $g$ of $\Gamma$ that is not in $T$, there is a unique cycle $Q=(T,g)$ containing $g$ and edges in $T$ only (\cite[Lemma 5.1]{Biggs}). In addition, as $g$ runs through the set $E\Gamma-T$, the $m-n+1$ elements $\xi_{T,g}$ form a basis of the cycle space of $G$ (see \cite[Theorem 5.2]{Biggs}) For an arbitrary graph, the result follows from applying this observation to each component.



\begin{prop}\label{prop:TC}
Let $C=(v_1,\dots,v_{\ell})$ be an ordered tuple of vertices that form a cycle in $G$ with edges $e_1=\{v_1,v_2\},\ldots, e_{\ell-1}=\{v_{\ell-1},v_{\ell}\}$, and $e_{\ell}=\{v_{\ell},v_1\}$. If $x \in \ker \delta_0^T$ has support $\{e_1,\dots,e_{\ell}\}$, then $x = c \cdot T(C)$ for some non-zero $c$.
\end{prop}

\begin{proof}

As in Definition \ref{def:TC}, denote $T(C)=\sum_{i=1}^{\ell}c_ie_i$. Let $x=\sum_{i=1}^{\ell}\alpha_ie_i$. As $\ker\delta_0^T x=0$, for each $j\in \{2,\ldots,\ell\}$, $[e_{j-1}:v_j]\alpha_{j-1}+[e_j:v_j]\alpha_j=0$ and therefore,
\begin{equation}\label{eq:coeffxL418}
    \alpha_j=-\alpha_{j-1}[e_{j-1}:v_j][e_j:v_j].
\end{equation}
Using \eqref{eq:coeffTC}, \eqref{eq:coeffxL418} and induction on $j$, we deduce that $\alpha_j=\alpha_1c_j$ for each $j\in \{1,\ldots,\ell\}$. This implies that $x=\alpha_1 T(C)$ and finishes our proof.
\end{proof}

\begin{prop}\label{rmk-tc-uniqueness}
If the ordered tuple of vertices $(v_1,\dots,v_{\ell})$ form a cycle in $G$, then    
\begin{equation*}
T(v_1,\dots,v_{\ell}) = T(v_i,\dots,v_{\ell},v_1,\dots,v_{i-1}) \text{ or } T(v_1,\dots,v_{\ell}) = - T(v_i,\dots,v_{\ell},v_1,\dots,v_{i-1}).
\end{equation*}
Thus, $T(v_1,\dots,v_{\ell}) \in \im \delta_1^T$ if and only if $T(v_i,\dots,v_{\ell},v_1,\dots,v_{i-1}) \in \im \delta_1^T$ for some $i$ (and so, for every $i$).
\end{prop}

\begin{proof}
For the first statement, it suffices to show that 
\begin{equation*}
T(v_1,\dots,v_{\ell}) = T(v_{\ell},v_1,\dots,v_{\ell-1}) \text{ or } T(v_1,\dots,v_{\ell}) = -T(v_{\ell},v_1,\dots,v_{\ell-1}).
\end{equation*}
Denote by $C=(v_1,\dots,v_{\ell})$ and $C'=(v_{\ell},v_1,\dots,v_{\ell-1}).$ Note that $T(C') \in \ker \delta_0^T$ and the support of $T(C')$ is the cycle $(v_1,\dots,v_{\ell})$ in $G$. The previous proposition implies that $T(C') = cT(C)$, where $c$ is the coefficient of the edge $\{v_1,v_2\}$ in $T(C'),$ and equals $+1$ or $-1$.
    
    This shows that $T(C) \in \{\pm T(C')\}.$ The second statement follows by observing that $T(v_1,\dots,v_{\ell})$ and $T(v_i,\dots,v_{\ell},v_1,\dots,v_{i-1})$ are non-zero scalar multiples of each other, so one of them belongs to the subspace $\im \delta_1^T$ if and only if both of them do. 
\end{proof}

The next two lemmas are the key ingredients in most of the results in this section. 

\begin{lemma}\label{prop:wheel4}
Let $G$ be a graph with four vertices $a,b,c,d$ forming a cycle. If vertex $e$ is a common neighbor of $a,b,c,d$, then there are $x_1,x_2,x_3,x_4 \in \{\pm 1\}$ such that 
    \begin{gather}
        T((a,b,c,d)) = \delta_1^T (x_1 \{a,b,e\} + x_2 \{b,c,e\}  + x_3 \{c,d,e\}  + x_4 \{a,d,e\}).
    \end{gather}
    Consequently, $T((a,b,c,d))) \in \im \delta_1^T.$ 
\end{lemma}

\begin{proof}
We may assume that $a=\min\{a,b,c,d\}$. There are three possibilities to consider depending on the order on the set $\{a,b,c,d\}:$ when the two extremal values are adjacent ($a < b < c < d$ or $a < c < b < d$), and when the two extremal values are on the opposite ends of the cycle ($a < b <  c > d)$. In each of these three cases, there are five possibilities depending on the relative order of $e$ to $a,b,c,d-$ yielding fifteen cases in total. Below we demonstrate the existence of $x_1,x_2,x_3,x_4$ satisfying the equation above. In Figure \ref{fig:enter-label2}, we use symbols $2,4,6,8$ for $a,b,c,d$ based on the order on $\{a,b,c,d\}$ so that $2 < 4 < 6 < 8$ holds. For $e$, we use one of the symbols $1,3,5,7,9$ depending on its relative order to the remaining symbols. The $x_1,x_2,x_3,x_4$ values are shown at the center of the corresponding triangle. Also, the red and blue edges respectively denote the $+1$ and $-1$ coefficients of $\ker \delta_0^T$ elements. For example, the top - left figure expresses the equation:
    \begin{gather}
        \delta_1^T(\{1,2,4\} + \{1,4,8\}-\{1,6,8\}-\{1,2,6\}) = \{2,4\} + \{4,8\} - \{6,8\} - \{2,6\}.
    \end{gather}
    \begin{figure}[h!]
        \centering
        \input{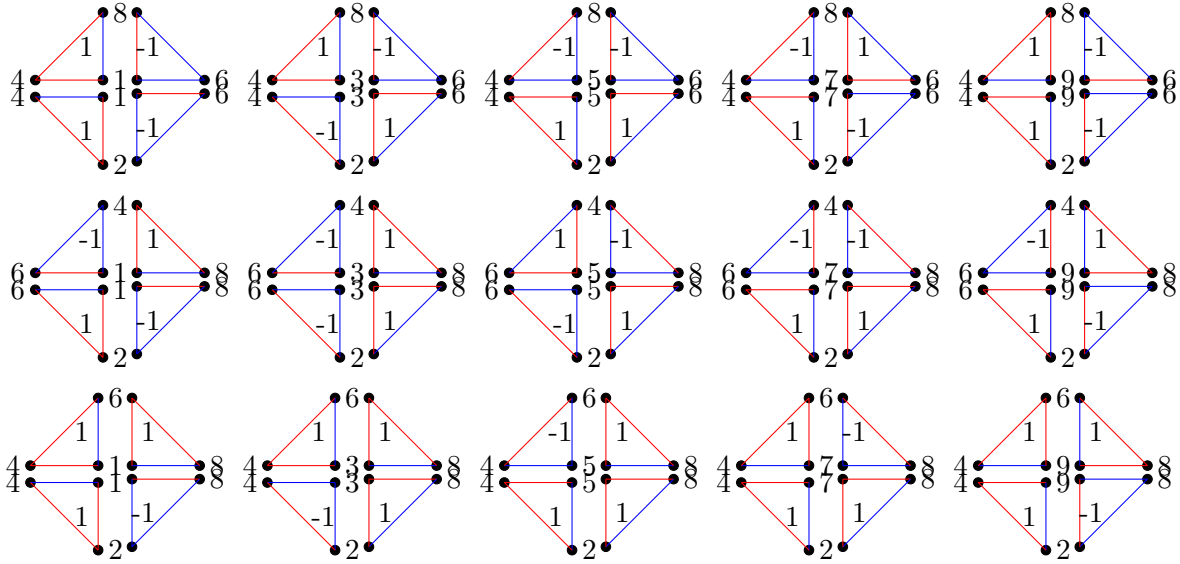}
        \caption{Decomposition of $T((a,b,c,d))$}
        \label{fig:enter-label2}
    \end{figure}
Since $\im \delta_1^T \subset \ker \delta_0^T,$ the above expression belongs to $\ker \delta_0^T.$ Moreover, since the support of this $\ker \delta_0^T$ element is the cycle $(a,b,c,d)$ in the graph, by Proposition \ref{rmk-tc-uniqueness}, we conclude that 
    \begin{gather}
        \delta_1^T(\{1,2,4\} + \{1,4,8\}-\{1,6,8\}-\{1,2,6\}) = \pm T(a,b,c,d).
    \end{gather}
    This concludes the proof.
\end{proof}

\begin{lemma}\label{prop:supp_red}
Let $G$ be a graph. Assume that $C=(a,b,c,d,v_5,\dots,v_{\ell})$ is a cycle of length four or more, and $a,b,c,d$ have a common neighbor $e$. Denote $e_1 = \{a,b\}, e_2 = \{b,c\}, e_3 = \{c,d\}$, and $c_1 = [e_1, T(C)], c_2 = [e_2, T(C)], c_3 = [e_3, T(C)].$ There are $c'_4,c'_5,c'_6,x_1, x_2\in \{-1,1\}$ such that
    \begin{gather}
        c_1e_1 + c_2e_2 + c_3e_3 + \delta_1^T(c'_4 \{a,b,e\} + c'_5 \{b,c,e\} + c'_6 \{c,d,e\}) = x_1 \{a,e\} + x_2 \{d,e\}.
    \end{gather}
\end{lemma}

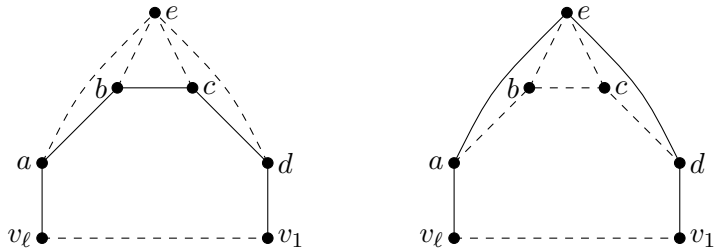
\begin{figure}[b!]
    \centering
    \begin{tikzpicture}
    \draw (0,0) -- (1,1) -- (2,1) -- (3,0);
    \draw (0,0) -- (0,-1);
    \draw (3,0) -- (3,-1);
    \draw[dashed] (0,-1) -- (3,-1);
    \filldraw[black] (0,0) circle (2pt) node[anchor=east]{$a$};
    \filldraw[black] (1,1) circle (2pt) node[anchor=east]{$b$};
    \filldraw[black] (2,1) circle (2pt) node[anchor=west]{$c$};
    \filldraw[black] (3,0) circle (2pt) node[anchor=west]{$d$};
    \filldraw[black] (0,-1) circle (2pt) node[anchor=east]{$v_{\ell}$};
    \filldraw[black] (3,-1) circle (2pt) node[anchor=west]{$v_1$};
    \filldraw[black] (1.5,2) circle (2pt) node[anchor=west]{$e$};
    \draw[dashed] (0,0) .. controls (0.5,1) .. (1.5,2);
    \draw[dashed] (3,0) .. controls (2.5,1) .. (1.5,2);
    \draw[dashed] (1,1) -- (1.5,2);
    \draw[dashed] (2,1) -- (1.5,2);
\end{tikzpicture}
\begin{tikzpicture}
    \draw[white] (1,1) -- (2,1);
\end{tikzpicture}
\begin{tikzpicture}
    \draw[dashed] (0,0) -- (1,1) -- (2,1) -- (3,0);
    \draw (0,0) -- (0,-1);
    \draw (3,0) -- (3,-1);
    \draw[dashed] (0,-1) -- (3,-1);
    \filldraw[black] (0,0) circle (2pt) node[anchor=east]{$a$};
    \filldraw[black] (1,1) circle (2pt) node[anchor=east]{$b$};
    \filldraw[black] (2,1) circle (2pt) node[anchor=west]{$c$};
    \filldraw[black] (3,0) circle (2pt) node[anchor=west]{$d$};
    \filldraw[black] (0,-1) circle (2pt) node[anchor=east]{$v_{\ell}$};
    \filldraw[black] (3,-1) circle (2pt) node[anchor=west]{$v_1$};
    \filldraw[black] (1.5,2) circle (2pt) node[anchor=west]{$e$};
    \draw (0,0) .. controls (0.5,1) .. (1.5,2);
    \draw (3,0) .. controls (2.5,1) .. (1.5,2);
    \draw[dashed] (1,1) -- (1.5,2);
    \draw[dashed] (2,1) -- (1.5,2);
\end{tikzpicture}
    \caption{Replacing three consecutive edges in $T(C)$ with two edges}
    \label{fig:support-reduction}
\end{figure}

\begin{proof}    
We define
\begin{gather}\label{eq:c456}
        c'_4 = -c_1[\{a,b,e\} : \{a,b\}] = -c_1[\{a,b,e\} : e_1],\\
        c'_5 = -c_2[\{b,c,e\}:\{b,c\}] = -c_2[\{b,c,e\}:e_2],\\
        c'_6 = -c_3[\{c,d,e\}:\{c,d\}] = -c_3[\{c,d,e\}:e_3].
\end{gather}
From the definition of $T(C)$, we know that $c_1,c_2,c_3 \in \{-1,1\}$, and consequently, $c'_4,c'_5,c'_6 \in \{-1,1\}$. 
    
For simplicity, denote 
\begin{equation}\label{eq:S}
S=c_1e_1 + c_2e_2 + c_3e_3+ \delta_1^T(c'_4 \{a,b,e\} + c'_5 \{b,c,e\} + c'_6\{c,d,e\}). 
\end{equation}
Recall that $[e_1,S]$ denotes the coefficient of $e_1=\{a,b\}$ in $S$. Using \eqref{eq:c456}, we get that
\begin{align*}
        [e_1, S]&= [e_1,c_1e_1 + c'_4 \delta_1^T(\{a,b,e\})] \\
        &=[e_1,c_1e_1 -c_1[\{a,b,e\}:e_1]\ \cdot\  \delta_1^T(\{a,b,e\})]\\
        &=[e_1,c_1e_1 -c_1[\{a,b,e\}:e_1]\ \cdot \ [e_1, \delta_1^T(\{a,b,e\})]e_1]\\
        &= [e_1,c_1e_1 -c_1[\{a,b,e\}:e_1]\ \cdot \ [\{a,b,e\}:e_1]e_1]\\ 
        &= [e_1, c_1e_1-c_1e_1] = 0.
    \end{align*}
Hence, there are no $e_1$ terms in the expansion of $S$. By similar arguments (the proofs are almost identical to the previous one), one can show that 
    \begin{gather}
        [e_2, S] = 0 \text{ and } [e_3, S] = 0.
    \end{gather}

We will now prove that 
    \begin{gather}
        [\{b,e\}, S] = 0.
    \end{gather}

From \eqref{eq:epsFF'}, if $\{x,y,z\}$ is any triangle in $G$, then
\begin{equation*}
[\{x,y,z\}:\{x,y\}]\cdot [\{x,y,z\}:\{x,z\}]=-[\{x,y\}:\{x\}][\{x,z\}:\{x\}].
\end{equation*}
Since $[\{x,y,z\}:\{x,y\}]\cdot [\{x,y,z\}:\{x,z\}]\cdot [\{x,y,z\}:\{y,z\}]=-1$, we deduce that
\begin{gather}
[\{x,y,z\}:\{y,z\}]=-[\{x,y\}:x] \cdot [\{x,z\}:x] = [\{x,y\}:y] \cdot [\{x,z\}:z]
\end{gather}
  
Using \eqref{eq:S} and the fact that $\{b,e\}$ is not in the triangle $\{c,d,e\}$, we get that
\begin{align*}
    [\{b,e\},S]&=[\{b,e\},\delta_1^T(c'_4\{a,b,e\}+c'_5\{b,c,e\})]=c'_4[\{b,e\},\delta_1^T\{a,b,e\}]+c'_5[\{b,e\},\delta_1^T\{b,c,e\}]\\
    &=c'_4[\{a,b,e\}:\{b,e\}]+c'_5[\{b,c,e\}:\{b,e\}]\\
    &=-c_1[\{a,b,e\} : \{a,b\}][\{a,b,e\}:\{b,e\}]-c_2[\{b,c,e\}:\{b,c\}][\{b,c,e\}:\{b,e\}]\\
    &=c_1[\{a,b,e\}:\{a,e\}]+c_2[\{b,c,e\}:\{c,e\}]\\
    &=c_1[\{a,b,e\}:\{a,e\}]-c_1[\{a,b\}:\{b\}]\cdot [\{b,c\}:\{b\}]\cdot [\{b,c,e\}:\{c,e\}]\\
    &=c_1[\{a,b,e\}:\{a,e\}]-c_1[\{a,b\}:\{b\}]\cdot [\{b,c\}:\{b\}]\cdot [\{b,c\}:\{b\}]\cdot [\{b,e\}:\{b\}]\\
    &=c_1[\{a,b,e\}:\{a,e\}]-c_1[\{a,b\}:\{b\}][\{b,e\}:\{b\}]\\
    &=c_1[\{a,b,e\}:\{a,e\}]-c_1[\{a,b,e\}:\{a,e\}]\\
    &=0.
\end{align*}

One can show that
\begin{equation}
[\{c,e\}, S] = [\{c,e\}, \delta_1^T(c'_5\{b,c,e\} + c'_6\{c,d,e\}] = 0
\end{equation}
by repeating the above argument with the triangles $\{a,b,e\}$ and $\{b,c,e\}$ replaced by the triangles $\{b,c,e\}$ and $\{c,d,e\}$ respectively. 

We conclude the proof by observing the following:
\begin{enumerate}
        \item the support of $S$ cannot contain any edge outside those of the triangles $\{a,b,e\}, \{b,c,e\},$ and $\{c,d,e\}$,
        \item the terms $\{a,b\},\{b,c\}, \{c,d\}, \{b,e\}, \{c,e\}$ vanish in $S$, and,
        \item the terms $\{a,e\}$ and $\{d,e\}$ appear with coefficients $c_4[\{a,b,e\}: \{a,e\}],$ and $c_6[\{c,d,e\}:\{d,e\}],$ each of which is in $\{-1,1\}$.
    \end{enumerate} 

\end{proof}

We are now ready to present the main theorem of the section.

\begin{lemma} \label{prop:cycle-cut}
Let $C=(v_1,v_2,\dots,v_{\ell})$ be an ordered tuple of vertices of a cycle in $G$ with ${\ell} \geq 4$ and $v_1$ is adjacent to $v_i$ for some $3 \leq i \leq {\ell}-1$. Then
\begin{gather}
        T(C) = T(C_1) + [\{v_i, v_{i+1}\}, T(C)]\ T(C_2), 
\end{gather}
where $C_1=(v_1,v_2,\dots,v_i)$ and $C_2=(v_i,v_{i+1},\dots,v_{\ell},v_1).$
\end{lemma}
\begin{proof}

    
    


Denote 
\begin{align*}
    T(C)&=c_1 \{v_1,v_2\}+c_2\{v_2,v_3\}+\ldots +c_{i-1}\{v_{i-1},v_{i}\}+c_i\{v_i,v_{i+1}\}+\ldots +c_{\ell-1}\{v_{\ell-1},v_{\ell}\}+c_{\ell}\{v_{\ell},v_1\},\\
    T(C_1)&=c'_1\{v_1,v_2\}+c'_2\{v_2,v_3\}+\ldots +c'_{i-1}\{v_{i-1},v_{i}\}+c'_i\{v_i,v_1\},\\
    T(C_2)&=c''_i\{v_i,v_{i+1}\}+\ldots +c''_{\ell-1}\{v_{\ell-1},v_{\ell}\}+c''_{\ell}\{v_{\ell},v_1\}+c''_1\{v_1,v_i\}.
\end{align*}
where $c_1=c'_1=c''_i=1$ and 
\begin{align*}
    c_j&=-c_{j-1}[\{v_{j-1},v_j\}:v_j][\{v_j,v_{j+1}\}:v_j],  \forall j\in \{2,\ldots,\ell\},\\
    c'_j&=-c'_{j-1}[\{v_{j-1},v_j\}:v_j][\{v_j,v_{j+1}\}:v_j], \forall j\in \{2,\ldots,i\},\\
    c''_j&=-c''_{j-1}[\{v_{j-1},v_j\}:v_j][\{v_j,v_{j+1}\}:v_j], \forall j\in \{i+1,\ldots,\ell\}.
\end{align*}
Therefore, using induction on $j$, one can show that $c_j=c'_j$, for any $j\in \{1,\ldots,i-1\}$ and $c_j=c_i\cdot c''_j$, for any $j\in \{i,\ldots,\ell-1\}$.



We observe that $T(C) \in \ker \delta_0^T,$ and so,
\begin{align*}
    0 &= [v_i, \delta_0^T(T(C))] = [v_i, \delta_0^T(c_{i-1}\{v_{i-1},v_i\} + c_i\{v_i,v_{i+1}\})]\\
    &=  [v_i, \delta_0^T(c_{i-1}\{v_{i-1},v_i\})] + [v_i, \delta_0^T( c_i\{v_i,v_{i+1}\})]\\
    &=c_{i-1}[\{v_{i-1},v_i\}:v_i] + c_i[\{v_i,v_{i+1}\}:v_i].
\end{align*}
That is, 
\begin{gather*}
    c_{i-1}[\{v_{i-1},v_i\}:v_i] =- c_i[\{v_i,v_{i+1}\}:v_i], \text{ or,}\\
    c_{i-1}c_i[\{v_{i-1},v_i\}:v_i][\{v_i,v_{i+1}\}:v_i]= -1,
\end{gather*}
since each of the terms is in $\{-1,1\}.$ Similarly, since $T(C_2) \in \ker \delta_0^T$ and the only  two edges in $C_2$ incident with the vertex $v_i$ are $\{v_1,v_i\}$ and $\{v_i,v_{i+1}\},$ we can show that:
\begin{gather*}
    0 = [v_i, T(C_2)] = c_{i}''[\{v_i,v_{i+1}\}:v_i] + c_1''[\{v_1,v_i\}:v_i].
\end{gather*}
Since each of the two summands is in $\{-1,1\},$ we have
\begin{gather*}
    c_{i}''[\{v_i,v_{i+1}\}:v_i] = -c_1''[\{v_1,v_i\}:v_i], \text{ or,}\\
    c_1'' = -c_{i}''[\{v_1,v_i\}:v_i]\ [\{v_i,v_{i+1}\}:v_i].
\end{gather*}
The coefficient of the edge $\{v_1,v_i\}$ in the expansion of $T(C_1)+c_iT(C_2)$ equals $c_i'+c_ic_1''$, where terms $c_i', c_ic_1'' \in \{-1,1\}$. We have that
\begin{align*}
    c_i'(c_ic_1'')&= \big(-c_{i-1}'[\{v_1,v_i\}:v_i] \ [\{v_{i-1},v_i\}:v_i]\big) \cdot \big( -c_ic_i''[\{v_1,v_i\}:v_i]\ [\{v_i,v_{i+1}\}:v_i]\big)\\
    &= \big(-c_{i-1}[\{v_1,v_i\}:v_i] \ [\{v_{i-1},v_i\}:v_i]\big) \cdot \big( -c_i[\{v_1,v_i\}:v_i]\ [\{v_i,v_{i+1}\}:v_i]\big)\\
    &=c_{i-1}c_i[\{v_{i-1},v_i\}, v_i] \ [\{v_i, v_{i+1}\}:v_i] = -1.
\end{align*}
That is, $c_i'$ and $c_ic_1''$ have opposite signs, and therefore $c_i'+c_ic_1''=0.$ This implies that term $\{v_1,v_{i}\}$ vanishes in the expansion of $T(C_1)+c_iT(C_2).$ We conclude that the support of the $\ker \delta_0^T$ element $T(C_1)+c_iT(C_2)$ forms the cycle $C$. By Prop \ref{prop:TC}, there is some nonzero constant $c$ such such that $T(C) = c(T(C_1)+c_iT(C_2)).$ Since in both the expressions $T(C)$ and $T(C_1)+c_iT(C_2)$ the coefficient of the edge $\{v_1,v_2\}$ is $1$, we conclude that $c=1,$ or, $T(C) = T(C_1)+c_iT(C_2).$ 
\end{proof}

\begin{corollary}\label{prop:two-triagles}
    If the four vertices $(a,b,c,d)$ form a cycle in $G$ such that either $a \sim c$ or $b \sim d$, then, $T(a,b,c,d) \in \im \delta_1^T.$ 
\end{corollary}

\begin{corollary}\label{cor:TCinduced}
The elements of the set 
\begin{equation*}
    \{T(C): C=(v_1,\dots,v_{\ell}) \text{ a ordered tuple of vertices in } G \text{ such that } v_1,\dots,v_{\ell} \text{ is an induced cycle} \}
\end{equation*}   
span $\ker \delta_0^T$.
\end{corollary}

\begin{theorem}\label{prop:cycle-decomp}
If $G$ is a graph with the property that any induced cycle (of length at least four) has four consecutive vertices that have a common neighbor, then the clique complex of $G$ satisfies $\im \delta_0 = \ker \delta_1$. 
\end{theorem}

\begin{proof}
Because $\frac{\ker \delta_1}{\im \delta_0} \cong \frac{\ker \delta_0^T}{\im \delta_1^T}$ (see \cite[Theorem 5.3]{Lim}), it is enough to prove that $\ker \delta_0^T = \im \delta_1^T$. By Corollary \ref{cor:TCinduced}, it suffices to show that if $C=(v_1,\dots,v_{\ell})$ is an induced cycle in $G$, then $T(C) \in \im \delta_1^T$. We prove this statement by strong induction on $\ell$. We first prove the base cases ${\ell}=3,4,5$ separately.

If $\ell =3$, then the support of $\delta_1^T (\{v_1,v_2,v_3\})\in \ker \delta_0^T$ consists of the edges of the cycle $C$. By Prop \ref{prop:TC}, there is some $c\in \{-1,1\}$ such that $\delta_1^T(\{v_1,v_2,v_3\}) = cT(C)$ which implies that $T(C)\in \im\delta_1^T$.

If $\ell=4$, then by using the hypothesis and Prop \ref{prop:wheel4}, we deduce that $T(C) \in \im \delta_1^T$. 

Let $\ell\geq 5$ and assume that $T(C') \in \im \delta_1^T$ for any cycle $C'$ of length at most $\ell-1$. Consider an induced cycle $C=(v_1,\dots,v_{\ell})$ of length $\ell$. By the hypothesis and Prop \ref{prop:supp_red}, there is $S \in \im \delta_1^T$ such that the $T(C)+S=T(C')$ or $T(C)+S=-T(C')$, where $C'$ is a cycle in $G$ with length $\ell-1$. By the inductive hypothesis, $T(C')\in \im \delta_1^T$ and $-T(C')\in \im \delta_1^T$, and consequently, $T(C) \in \im \delta_1^T$.
\end{proof}
The following result is corollary of Theorem \ref{prop:cycle-decomp} and was previously obtained by Meshulam \cite{Meshulam} (see also the case $k=1$ in Kahle \cite[Thm 3.1]{Kahle}).

\begin{corollary}[Meshulam \cite{Meshulam}] \label{cor:Meshulam}
Let $G$ be a graph with the property that any four vertices of $G$ have a common neighbor. Then the clique complex of $G$ satisfies that $\im \delta_0 = \ker \delta_1.$
\end{corollary}

For $n,t\in \N$, the Kneser graph $K(n,t)$ is the graph whose vertices are the $t$-subsets of $\{1,\ldots,n\}$, where two $t$-subsets are adjacent if and only if they are disjoint.


The Kneser graphs $K(8,2)$ and $K(9,2)$ satisfy the hypothesis of Theorem \ref{prop:cycle-decomp}, and hence has a trivial first cohomology group, even though it does not satisfy the condition in Corollary \ref{cor:Meshulam}. We show the first statement for $K(8,2)$ below, and the proof for $K(9,2)$ proceeds similarly. The induced $4,5,$ and $6$ cycles in $K(8,2)$ have the form of the first three graphs in Figure \ref{fig:k-8-2-cycles}, up-to a permutation of the symbols. For each of these cycles, the vertex $\{7,8\}$ is a common neighbor of all the vertices appearing in the respective cycle. Moreover, an induced path of length $6$ in $K(8,2)$ have the form of the right-most graph in Figure \ref{fig:k-8-2-cycles}, up-to relabeling of the symbols $[n].$ This shows that there are no induced paths of length $6,$ and so, no induced cycle of length more than $6$ exists. We conclude that $K(8,2)$ satisfies the hypothesis of Theorem \ref{prop:cycle-decomp}. Finally, the four vertices $\{1,2\}, \{3,4\}, \{5,6\}, \{7,8\}$ have no common neighbor, which shows that $K(8,2)$ does not satisfy the condition in Corollary \ref{cor:Meshulam}. 
\begin{figure}[htbp]
    \centering
    \scalebox{0.7}{
    \begin{tikzpicture}
        \filldraw[black] (1,1) circle (2pt) node[anchor=south]{$\{1,2\}$};
        \filldraw[black] (0,0) circle (2pt) node[anchor=east]{$\{3,4\}$};
        \filldraw[black] (2,0) circle (2pt) node[anchor=west]{$\{3,5\}$};
        \filldraw[black] (1,-1) circle (2pt) node[anchor=north]{$\{1,6\}$};
        \draw[black] (1,1) -- (0,0) -- (1,-1) -- (2,0) -- (1,1);
    \end{tikzpicture}
    \hspace{10pt}
    \begin{tikzpicture}
        \filldraw[black] (1,1) circle (2pt) node[anchor=south]{$\{1,2\}$};
        \filldraw[black] (0,0) circle (2pt) node[anchor=east]{$\{3,4\}$};
        \filldraw[black] (2,0) circle (2pt) node[anchor=west]{$\{3,5\}$};
        \filldraw[black] (0,-1) circle (2pt) node[anchor=east]{$\{1,5\}$};
        \filldraw[black] (2,-1) circle (2pt) node[anchor=west]{$\{4,2\}$};
        \draw[black] (1,1)--(0,0)--(0,-1)--(2,-1)--(2,0)--(1,1);
    \end{tikzpicture}
    \hspace{10pt}
    \begin{tikzpicture}
        \filldraw[black] (1,1) circle (2pt) node[anchor=south]{$\{1,2\}$};
        \filldraw[black] (0,0) circle (2pt) node[anchor=east]{$\{3,4\}$};
        \filldraw[black] (2,0) circle (2pt) node[anchor=west]{$\{3,5\}$};
        \filldraw[black] (0,-1) circle (2pt) node[anchor=east]{$\{1,5\}$};
        \filldraw[black] (0,-2) circle (2pt) node[anchor=east]{$\{3,2\}$};
        \filldraw[black] (2,-1) circle (2pt) node[anchor=west]{$\{4,1\}$};
        \draw[black] (1,1)--(0,0)--(0,-1)--(0,-2)--(2,-1)--(2,0)--(1,1);
    \end{tikzpicture}
    \begin{tikzpicture}
        \filldraw[black] (1,1) circle (2pt) node[anchor=south]{$\{1,2\}$};
        \filldraw[black] (0,0) circle (2pt) node[anchor=east]{$\{3,4\}$};
        \filldraw[black] (2,0) circle (2pt) node[anchor=west]{$\{3,5\}$};
        \filldraw[black] (0,-1) circle (2pt) node[anchor=east]{$\{1,5\}$};
        \filldraw[black] (2,-1) circle (2pt) node[anchor=west]{$\{4,1\}$};
        \draw[black] (2,-1)--(2,0)--(1,1)--(0,0)--(0,-1);
    \end{tikzpicture}    }
    \caption{From the left, general forms of induced cycles of length $4,5,6$, and induced paths of maximum length in the Kneser graph $K(8,2)$, up-to relabeling of the symbols}
    \label{fig:k-8-2-cycles}
\end{figure}
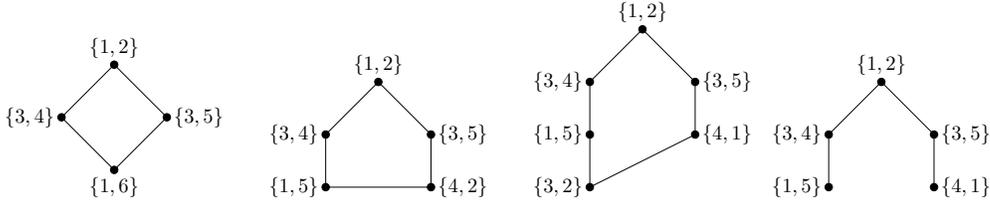

Another class of (not necessarily strongly regular) graphs satisfying the hypothesis in Theorem \ref{prop:cycle-decomp} is the class of chordal graphs. A chordal graph is a graph with no induced cycle of length more than three. A subclass of chordal graphs is the class of interval graphs which are defined as follows. If $S$ is a finite set of nonempty, closed, bounded intervals, then the interval graph $I(S)$ is the graph that has vertex set $S$, where two vertices of $I(S)$ are adjacent if the corresponding intervals have nonempty intersection. An interval graph is chordal. If $G$ is a chordal graph with at least $4$ vertices, then $G$ satisfies the hypothesis in Theorem \ref{prop:cycle-decomp}. In addition, if $G$ has diameter three or more, then $G$ does not satisfy the condition in Corollary \ref{cor:Meshulam}.
Note that there are interval graphs of arbitrary large diameters.

\begin{corollary}
For $n \geq 5t,$ the clique complex of the Kneser graph $K(n,t)$ satisfies $\im \delta_0 = \ker \delta_1$.
\end{corollary}

Let $q$ be a prime power with $q \equiv 1 \pmod 4$. The Paley graph Paley$(q)$ is defined as the graph with vertex set the elements of the finite field $\mathbb F_q$ and two vertices $x$ and $y$ are adjacent if and only if $x-y$ is a non-zero square. 
 This is a strongly regular graph with parameters $\left(q,\frac{q-1}{2},\frac{q-5}{4},\frac{q-1}{4}\right)$, see \cite[Section 9.1.2]{BH}. For sufficiently large values of $q$, the Paley$(q)$ satisfies the hypothesis of Corollary \ref{cor:Meshulam}, see \cite[Thm. 8.1]{Babai}. 

\begin{corollary}
   For $q \equiv 1 \pmod 4$ prime power and $q > 4096$, the clique complex of the Paley Graph on $q$ vertices satisfies that $\im \delta_0 = \ker \delta_1.$
\end{corollary}

We apply the last result to general strongly regular graphs and find two sufficient conditions for the first cohomology group to be trivial.

\begin{theorem}
If $G$ is a $(v,k,\lambda,\mu)$-SRG such that $2\lambda + \mu + 2 > 2k$, then the clique complex of $G$ satisfies $\ker \delta_1 = \im \delta_0.$
\end{theorem}
\begin{proof}
Let $a,b,c,$ and $d$ be four consecutive vertices in some induced cycle in $G$. Therefore, $a,c \in \Gamma(b)$, $b\not\sim d$ and $a \not \sim c$. Denote $X = \Gamma(a) \cap \Gamma(b)$ and $Y = \Gamma(c) \cap \Gamma(b)$. Because $a, c\notin X \cup Y$, $X\cup Y\subset \Gamma(b)$ and $a,c\in \Gamma(b)$, we get that $|X \cup Y| \leq k-2$. Thus, $|X \cap Y| = |X| + |Y| - |X \cup Y| \geq 2\lambda-k+2$. Because $b\not\sim d$, $|\Gamma(b)\cap \Gamma(d)|=\mu$. Thus, $|X \cap Y| + |\Gamma(b) \cap \Gamma(d)| \geq 2\lambda-k+2 + \mu >k$, where the last inequality follows from our hypothesis. Since $X\cap Y\subset \Gamma(b), \Gamma(b)\cap \Gamma(d)\subset \Gamma(b)$ and $|\Gamma(b)|=k$, we deduce that $|X \cap Y \cap \Gamma(d)| > 0$. This shows that $a, b, c,$ and $d$ have a common neighbor and completes the proof.
\end{proof}

For prime power $q$ and $r\geq 2$, the symplectic graph $Sp(2r,q)$, see \cite{TANG},  is strongly regular graph with parameters
$$\left(\frac{q^{2r}-1}{q-1},q^{2r-1},q^{2r-1}-q^{2r-2},q^{2r-1}-q^{2r-2}\right).$$ 
When $q\geq 3$ and $r\geq 2$, the graph $Sp(2r,q)$ satisfies the hypothesis of the previous theorem. There are other strongly regular graphs with this property, such as the ones with parameters 
\begin{align*}
  (36,25,16,20), (40,27,18,18), & (45,32,22,24), (49,36,25,30), \\
  (50,42,35,36), (56,45,36,36), & (64,45,32,30), (81,56,37,42).
\end{align*}

\begin{theorem}\label{thm:srg5cycle}
If $G$ is a $(v,k,\lambda,\mu)$-SRG such that 
    \begin{enumerate}[(a)]
        \item for any two non-adjacent vertices $x$ and $y$ of $G$, the induced subgraph on the vertex set $\Gamma(x) \cap \Gamma(y)$ is connected, 
        \item any induced $5$-cycle in $G$ has four vertices with a common neighbor in $G$,
    \end{enumerate}
    then the clique complex of $G$ satisfies $\ker \delta_1 = \im \delta_0$.
\end{theorem}
\begin{proof}
We will prove that if $C$ is any cycle in $G$, then $T(C) \in \im \delta_1^T$ by induction on the length of $C$. We first prove the bases cases when the length of $C$ is $3,4,$ or $5$. 

If $C=(v_1,v_2,v_3)$ is a cycle of length three, then by Prop \ref{prop:TC} $T(C)=\delta_1^T(\{v_1,v_2,v_3\})$ or $T(C)=-\delta_1^{T}(\{v_1,v_2,v_3\})$. In either case, $T(C)\in \im\delta_1^T$. 




We now prove that if $C=(a,b,c,d)$ is a cycle of length $4$, then $T(C) \in \im \delta_1^T$. 



Because $b,d\in \Gamma(a)\cap \Gamma(c)$, there is a path between $b$ and $d$ in the induced subgraph of $G$ on the vertex set $\Gamma(a) \cap \Gamma(c)$. Denote by $D_{a,c}(b,d)$ the length of a shortest path between $b$ and $d$ in $\Gamma(a) \cap \Gamma(c)$. Similarly, denote by $D_{b,d}(a,c)$ the length of a shortest path between $a$ and $c$ in $\Gamma(b) \cap \Gamma(d)$. We also define
\begin{equation*}
D(a,b,c,d) = \min\{D_{a,c}(d,b), D_{b,d}(a,c)\}.
\end{equation*}

%

We now prove that if $(a,b,c,d)$ is any cycle in $G$, then $T(a,b,c,d) \in \im \delta_1^T$ by strong induction on $D(a,b,c,d)$. 

If $(a,b,c,d)$ is a cycle in $G$ with $D(a,b,c,d)=1$, then either $a \sim c$ or $b \sim d$. In either case, Corollary \ref{prop:two-triagles} implies that $T(a,b,c,d) \in \im \delta_1^T$. 

If $D(a,b,c,d)=2$, then there is a common neighbor of the vertices $a, b, c$, and $d$. By Prop \ref{prop:wheel4}, we deduce that $T(a,b,c,d) \in \im \delta_1^T$.

Let $\ell\geq 3$. Assume that if $C'=(a',b',c',d')$ is any cycle in $G$ with $D(a',b',c',d') \leq \ell-1,$ then $T(a',b',c',d') \in \im \delta_1^T$.

Let $C=(a,b,c,d)$ be a cycle with $D(a,b,c,d)=\ell$. Without loss of generality, assume that there is a path $v_1=b,v_2,\dots,v_{\ell+1}=d$ in $\Gamma(a) \cap \Gamma(c)$ and $D_{b,d}(a,c)\geq \ell$.

We note that $C'=(a,b,c,v_{\ell})$ is a cycle in $G$ and $D(a,b,c,v_{\ell}) \leq \ell-1$. Also $C_1=(a,v_{\ell},d)$ and $C_2 = (c,v_{\ell},d)$ are cycles in $G$. Set 
\begin{align*}
    x_1 &= -[\{a,d\}, T(C)] \cdot [\{a,d\}, T(C_1)],\\
    x_2 &= -[\{c,d\}, T(C)] \cdot [\{c,d\}, T(C_1)],\\
    S &= x_1 T(C_1) + x_2T(C_2).
\end{align*}

By the choices of $x_1$ and $x_2$, the terms $\{a,d\}$ and $\{c,d\}$ vanish in the expansion of $T(C)+S$. Thus, the support of $T(C)+S$ is contained in $\{\{a,b\}, \{b,c\}, \{c,v_{\ell}\}, \{v_{\ell},a\}, \{v_{\ell},v_{\ell+1}\}\}.$ Since $T(C)+S \in \ker \delta_0^T$ and the only edge in this set incident with the vertex $v_{\ell+1}$ is $\{v_{\ell},v_{\ell+1}\},$ we conclude that the term $\{v_{\ell},v_{\ell+1}\}$ also vanishes in $T(C)+S.$ Finally, each of the remaining four edges appears with a nonzero coefficient in the expansion of $T(C)+S$. Thus, the $\ker \delta_0^T$ element $T(C)+S$ has support comprising of the edges of the cycle $C'=(a,b,c,v_{\ell})$. By Prop \ref{prop:TC}, we deduce that $T(C)+S=cT(C')$ for some nonzero $c$. As $D(C') < \ell$, by the inductive hypothesis, $T(C') \in \im \delta_1^T$. Consequently, $T(C) = cT(C')-S \in \im \delta_1^T$.  

Consider now a cycle $C=(a,b,c,d,e)$ of length $5$.  If $C$ is not induced, then, without loss of generality, we may assume that $a \sim d.$ By Prop \ref{prop:cycle-cut}, there is $c \in \{-1,1\}$ such that 
\begin{equation*}
    T(C) = T(a,b,c,d) + cT(d,e,a).
\end{equation*} 
By the work done earlier in this proof, $T(a,b,c,d)\in \im \delta_1^T$ and $T(d,e,a)\in \im \delta_1^T$. Thus, $T(C) \in \im \delta_1^T$. 

If $C$ is induced, then, by our hypothesis, there is some vertex $f$ in $G$ such that $f$ is incident with four vertices of the cycle. Without loss of generality, assume that $f$ is adjacent with the vertices $a, b, c,$ and $d$. By Prop \ref{prop:supp_red}, there is $S \in \im \delta_1^T$ such that the $\ker \delta_0^T$ element $T(C)+S$ has support $\{\{a,f\},\{f,d\},\{d,e\},\{a,e\}\}.$ By Prop \ref{rmk-tc-uniqueness}, there is $c'\in \{-1,1\}$ such that 
\begin{equation*}
T(C)+S = c'T(C'),
\end{equation*}
where $C'=(a,f,d,e)$ is a cycle of length $4$. By our previous case, we have $T(C') \in \im \delta_1^T$, and so, $T(C)= c'T(C')-S \in \im \delta_1^T$.

Let $\ell\geq 6$. Assume that for any cycle $C'$ of length $\ell -1$ or less, $T(C') \in \im \delta_1^T$. Let $C=(v_1,\dots,v_{\ell})$ be a cycle of length $\ell$. We will show that $T(C) \in \im \delta_1^T$. If $C$ is not induced, we may assume, without the loss of generality, that $v_1 \sim v_i$ for some $3 \leq i \leq \ell-2$. By Prop \ref{prop:cycle-cut}, there is $c'' \in \{-1,1\}$ such that
\begin{gather*}
    T(C) = T(C_1) + c''T(C_2),
\end{gather*}
where $C_1=(v_1,\dots,v_i)$ and $C_2=(v_i,v_{i+1},\dots,v_{\ell},v_1)$. Since both the cycles $C_1,C_2$ have length at most $\ell-1$, $T(C_1) \in \im \delta_1^T$ and $T(C_2) \in \im \delta_1^T$ from the induction hypothesis. Consequently, $T(C) \in \im \delta_1^T$. 

If $C$ is an induced cycle, then there is a vertex $x\notin \{v_1,\dots,v_{\ell}\}$ such that $v_1 \sim x$ and $v_4 \sim x$ (since $v_1 \not \sim v_4$). 
We denote
\begin{gather*}
    C_1 = (v_1,x,v_4,v_3,v_2), \text{ and } C_2 = (v_1,x,v_4,v_5,\dots,v_{\ell}).
\end{gather*}
Consider
\begin{gather*}
    S := T(C_1)-T(C_2).
\end{gather*}
From the definition of $T$, we have that 
\begin{gather*}
    [\{v_1,x\}, T(C_1)] = [\{v_1,x\}, T(C_2)] =1.
\end{gather*}
From Definition \ref{def:TC}, we get that
\begin{align*}
    [\{x,v_4\}, T(C_1)]&= -[\{v_1,x\}, T(C_1)] \cdot [\{v_1,x\}:x] \cdot [\{x,v_4\}:v_4]= -[\{v_1,x\}:x] \cdot [\{x,v_4\}:v_4],\\
    [\{x,v_4\}, T(C_2)] &= -[\{v_1,x\}, T(C_2)] \cdot [\{v_1,x\}:x] \cdot [\{x,v_4\}:v_4]= -[\{v_1,x\}:x] \cdot [\{x,v_4\}:v_4].
\end{align*}
Therefore, both terms $\{v_1,x\}$ and $\{x,v_4\}$ vanish in the expansion of $T(C_1)-T(C_2)$. Since each of the remaining edges in $C_1$ and $C_2$ appear with coefficients $1$ or $-1$ in the expansion of $T(C_1)-T(C_2)$, by Prop \ref{prop:TC}, we conclude that there is some $c$ with 
\begin{gather*}
    T(C_1)-T(C_2) = cT(C).
\end{gather*}
Observe now that the length of the cycles $C_1$ and $C_2$ are $5$ and $\ell-1$ respectively. Hence, by the inductive hypothesis, both of the terms $T(C_1)$ and $T(C_2)$ are in $\im \delta_1^T.$ We conclude that $T(C) \in \im \delta_1^T$ as well, and the proposition follows.     
\end{proof}





A conference graph on $v$ vertices is a strongly regular graph with parameters $(v, \frac{v-1}{2}, \frac{v-5}{4},\frac{v-1}{4}).$ 

\begin{prop}\label{prop:confGammaxconn}
Let $G$ be a $\left(v, \frac{v-1}{2}, \frac{v-5}{4},\frac{v-1}{4}\right)$-SRG. Let $x$ be a vertex of $G$ and $\Gamma(x)$ be the subgraph of $G$ induced by the neighborhood of $x$. If $v>9$, then the subgraph $\Gamma(x)$ is connected.
\end{prop}

\begin{proof} Assume that $v>9$. Let $x\in V$. The subgraph $\Gamma(x)$ has $k=\frac{v-1}{2}$ vertices and is regular of valency $\lambda=\frac{v-5}{4}$. In order to show that $\Gamma(x)$ is connected, we prove that $\sigma<\frac{v-5}{4}$, where $\sigma$ is the second largest eigenvalue of the adjacency matrix of $\Gamma(x)$. Using \cite[Lemma 10.6.1]{GR}, we deduce that $\sigma\leq \frac{-1+\sqrt{v}}{2}$. As $v>9$, we get that $\frac{-1+\sqrt{v}}{2}<\frac{v-5}{4}$ which proves our assertion.
\end{proof}

Note that when $v\in \{5,9\}$, the previous result is not true. When $v=5$, the graph $G$ is the cycle on $5$ vertices and for any vertex $x$, $\Gamma(x)$ consists of two isolated vertices. When $v=9$, the graph $G$ is the $3\times 3$ grid or Hamming graph $H(2,3)$ and for any vertex $x$, $\Gamma(x)$ is a disjoint union of two edges.

\begin{theorem}\label{thm:conference-h1}
Let $v > 9.$ If $G$ is a $\left(v, \frac{v-1}{2}, \frac{v-5}{4},\frac{v-1}{4}\right)$-SRG such that, for any pair of nonadjacent vertices $x$ and $y$, $\Gamma(x) \cap \Gamma(y)$ is connected, then $\im  \delta_0 = \ker \delta_1$.
\end{theorem}

\begin{proof}

It suffices to show that if $C$ is any cycle in $G$, then $T(C) \in \im \delta_1^T$. We prove this by induction on the length of $C$. In the proof of Theorem \ref{thm:srg5cycle}, we observed that if the length of $C$ is $3,$ then $T(C) \in \im \delta_1^T$. Using the second hypothesis, we have also proved in Theorem \ref{thm:srg5cycle} that if $C$ is any cycle of length $4,$ then $T(C) \in \im \delta_1^T.$

Consider now any cycle $C=(v_1,v_2,v_3,v_4,v_5)$ in $G$. If $C$ is not induced, then, without loss of generality, we assume $v_1 \sim v_4$. As in the proof in Theorem \ref{thm:srg5cycle}, 
$T(C) = T(v_1,v_2,v_3,v_4) + cT(v_4,v_5,v_1)$, for some $c\in \{-1,1\}$.  By the previous paragraph, $T(v_1,v_2,v_3,v_4) \in \im \delta_1^T$ and $T(v_4,v_5,v_1) \in \im \delta_1$. Thus $T(C) \in \delta_1^T$. 

Assume now that $C$ is induced. We consider two cases:
\begin{enumerate}
        \item \textbf{Case  $\Gamma(v_1) \cap \Gamma(v_3) \cap \Gamma(v_4) \not = \emptyset$:} Let $f \in \Gamma(v_1) \cap \Gamma(v_3) \cap \Gamma(v_4)$. Denote the cycles 
        \begin{align*}
            C_1 = (v_1,v_2,v_3,f),\  C_2 = (v_3,f,v_4),\  C_3=(v_4,f,v_1,v_5),\ C_4 = (v_1,v_2,v_3,v_4,f).
        \end{align*}
    Let 
    \begin{align*}
        x_1 & =-[\{v_3,f\}, T(C_1)\}]\cdot [\{v_3,f\},T(C_2)], \text{ and,}\\
        S_1 & = T(C_1)+x_1T(C_2). 
    \end{align*}
    By construction, the term $\{v_3,f\}$ vanishes in the expansion of $S_1.$ Thus, the support of $S_1$ consists of the edges of the cycle $C_4,$ since each of the the remaining edges each appear with coefficients $1$ or $-1$. Since $S_1 \in \ker \delta_0^T,$ by Prop \ref{prop:TC}, there is some constant $c$ such that $S_1 = cT(C_4).$ But since the coefficient of the edge $\{v_1,v_2\}$ is $1$ in both $T(C_4)$ and $S_1$, we deduce that $c=1$ and $S_1 = T(C_4)$. 
    
    Define now
    \begin{align*}
        x_2 &= -[\{f,v_4\}, T(C_4)] \cdot [\{f,v_4\}, T(C_3)], \text{ and,} \\
        S_2 &=  T(C_4)+x_2T(C_3) \in \ker \delta_0^T.
    \end{align*}
    By construction, the term $\{f,v_4\}$ vanishes in the expansion of $S_2$. Thus, the support of $S_2$ comprises of the edges in the cycle $C$. This implies that there is $c'\in \{-1,1\}$ such that $S_2 = c'T(C).$ Furthermore,
    \begin{gather*}
        [\{v_1,v_2\}, S_2] = [\{v_1,v_2\}, T(C_1)] = 1 = [\{v_1,v_2\}, T(C)],
    \end{gather*}
    which implies that $c'=1$ and $S_2 = T(C)$. We have that
    \begin{gather*}
        T(C)=T(C_4)+x_2T(C_3)=T(C_1)+x_1T(C_2)+x_2T(C_3).
    \end{gather*}
    Note that each of the cycles $C_1,C_2,C_3$ has length $3$ or $4,$ and hence by our arguments above, $T(C_i) \in \im \delta_1^T$, for any $i\in \{1,2,3\}$. Thus, $T(C) \in \im \delta_1^T$ as well.

    \item \textbf{Case  $\Gamma(v_1) \cap \Gamma(v_3) \cap \Gamma(v_4) = \emptyset$:} Let $E_1 = \Gamma(v_1) \cap \Gamma(v_3)$ and $E_2 = \Gamma(v_1) \cap \Gamma(v_4).$ Then $v_2 \in E_1, v_5 \in E_2$. Because $\Gamma(v_1)\cap \Gamma(v_3)\cap \Gamma(v_4)=\emptyset$,  $E_1 \cup E_2$ is a partition of $\Gamma(a)$ with $|E_1|=|E_2| = \frac{v-1}{4}$. Since $\Gamma(v_1)$ is connected, there are $x \in E_1$ and $y \in E_2$ such that $x \sim y$. 

Let us define the cycles
    \begin{align*}
        C_1 &= (x,y,v_1), \ & C_2= (x,y,v_4,v_3),\\
        C_3 &= (v_1,x,v_3,v_4,y), \ & C_4 = (v_1,v_2,v_3,x),\\
        C5 &= (v_1,v_2,v_3,v_4,y), \ & C_6=(v_1,v_5,v_4,y).
    \end{align*}

Denote $S_1 = T(C_1)-T(C_2)$. Since
    \begin{gather*}
        [\{x,y\}, T(C_1)] = [\{x,y\}, T(C_2)] =1,
    \end{gather*}
    the term $\{x,y\}$ vanishes in the expansion of $S_1.$ Thus, the $\ker \delta_0^T$ element $S_1$ has support precisely the set of edges of the cycle $C_3=(v_1,x,v_3,v_4,y)$. By Prop \ref{prop:TC}, there is $c\neq 0$ such that 
    \begin{gather*}
        S_1 = cT(C_3).
    \end{gather*}

Denote
    \begin{align*}
        x_1 &= -[\{x,v_3\}, T(C_4)]\cdot [\{x,v_3\}, S_1]=-[\{x,v_3\}, T(C_4)]\cdot [\{x,v_3\}, cT(C_3)], \text{ and,}\\
        S_2 &= S_1 + x_1T(C_4)= cT(C_3) + x_1 T(C_4). 
    \end{align*}
    By construction, the term $\{x,v_3\}$ vanishes in $S_2$. If we denote by $A_1$ and $A_2$ the set of edges in the cycle $C_5=(v_1,v_2,v_3,v_4,y)$ and the support of $S_2$, respectively, then
    \begin{gather*}
        A_1 \subseteq A_2 \subseteq A_1 \cup \{\{v_1,x\}\}.
    \end{gather*}
    As $\{v_1,x\}$ is the only edge in $A_1 \cup \{\{v_1,x\}\}$ incident with the vertex $x$, the support of the $\ker \delta_0^T$ element $S_2$ cannot be $A_1 \cup \{\{v_1,x\}\}.$ Hence, $A_2=A_1$. and by Prop \ref{prop:TC}, there is $c'\neq 0$ such that
    \begin{gather*}
        S_2 = c'T(C_5).
    \end{gather*}
    Thus,
    \begin{gather*}
        c'T(C_5) = S_2 = cT(C_3)+x_1T(C_4).
    \end{gather*}

Finally, we denote
    \begin{align*}
        x_2 &= -[\{y,v_4\}, T(C_6)]\cdot [\{y,v_4\}, S_2]=-[\{y,v_4\}, T(C_6)]\cdot [\{y,v_4\}, c'T(C_5)], \text{ and,}\\
        S_3 &= S_2 + x_2T(C_6) = c'T(C_5) + x_2T(C_6).
    \end{align*}
    By construction, the term $\{y,v_4\}$ vanishes in the expansion of $S_3$. If we denote by $A_3$ and $A_4$ by the set of edges in the cycle $C=(v_1,v_2,v_3,v_4,v_5)$ and the support of $S_3$, respectively, then
    \begin{gather*}
        A_3 \subseteq A_4 \subseteq A_3 \cup \{\{v_1,y\}\}.
    \end{gather*}
    As $S_3 \in \ker \delta_0^T$ and the only edge in $A_3 \cup \{\{v_1,y\}\}$ that is incident with the vertex $y$ is $\{v_1,y\},$ we conclude that $\{v_1,y\} \not \in A_4,$ and so, $A_3 = A_4.$ By Prop \ref{prop:TC}, there is some nonzero $c''$ such that 
    \begin{gather*}
        S_3 = c''T(C).
    \end{gather*}
    Using the above equations, we obtain that
\begin{align*}
        T(C) &= \frac{1}{c''}S_3= \frac{1}{c''}(c'T(C_5))+\frac{x_2}{c''}T(C_6)=\frac{1}{c''}(cT(C_3))+\frac{x_1}{c''}T(C_4)+\frac{x_2}{c''}T(C_6)\\
        &= \frac{1}{c''}S_1+\frac{x_1}{c''}T(C_4)+\frac{x_2}{c''}T(C_6)\\
        &= \frac{1}{c''}T(C_1)-\frac{1}{c''}T(C_2)+\frac{x_1}{c''}T(C_4)+\frac{x_2}{c''}T(C_6).
\end{align*}
Since each of the cycles $C_1,C_2,C_4,C_6$ has length $3$ or $4,$ by the above discussions, we have $T(C_i) \in \im \delta_1^T, i\in \{1,2,4,6\},$ and therefore, $T(C) \in \im \delta_1^T$ as well.
\end{enumerate}

We have thus established that if $C$ is any cycle of length $3,4,$ or $5,$ then $T(C) \in \im \delta_1^T.$ Assume now that if $C'$ is any cycle of length at most $\ell-1$, then $T(C') \in \im \delta_1^T.$ Let $C=(v_1,\dots,v_{\ell})$ be a cycle of length $\ell \geq 6$. If $C$ is not induced, we may assume that $v_1 \sim v_i$ for some $3 \leq i \leq \ell-1.$ By Prop \ref{prop:cycle-cut}, there is $d\in \{-1,1\}$ such that
\begin{gather*}
        T(C) = T(C_1) + dT(C_2),
\end{gather*}
where $C_1=(v_1,\dots,v_i), C_2=(v_i,v_{i+1},\dots,v_{\ell}, v_1)$. As both of the cycles $C_1$ and $C_2$ have lengths at most $\ell-1$, by the inductive hypothesis, $T(C_1) \in \im \delta_1^T$ and $T(C_2) \in \im \delta_1^T$. Hence, $T(C) \in \im \delta_1^T$.
    
On the other hand, if $C$ is induced, then $v_4 \not \sim v_1$. Since $G$ is strongly regular, there is some vertex $x$ in $G$, $x \notin \{v_1,\ldots,v_{\ell}\}$ such that $v_1 \sim x$ and $x \sim v_4$. We have proved in Theorem \ref{thm:srg5cycle} that there is some nonzero constant $d'$ such that
    \begin{gather*}
        d'T(C) = T(C'_1) - T(C'_2),
    \end{gather*}
    where $C'_1=(v_1,x,v_4,v_3,v_2)$ and $C'_2=(v_1,x,v_4,v_5,\dots,v_{\ell})$. Both cycles $C'_1$ and $C'_2$ have lengths at most $\ell-1$ and by the inductive hypothesis, $T(C'_1) \in \im \delta_1^T$ and $T(C'_2) \in \im \delta_1^T,$ and thus, $T(C) \in \im \delta_1^T$. By the induction principle, $T(C) \in \im \delta_1^T$ for any cycle $C$, and the proof concludes.
\end{proof}
\color{black}
We end the section by providing some examples of strongly regular graphs that satisfy the hypothesis in Theorem \ref{thm:conference-h1} and hence have a trivial first cohomology group. By computer, we have verified that Paley$(q)$ satisfies this hypothesis for $q \in \{17,25,29,37,41,49,53,61,73,81,89,97\}$, while of these, only Paley$(89)$ and Paley$(97)$ satisfy the condition in Corollary \ref{cor:Meshulam}. Of the $15$ conference graphs on $25$ vertices having the parameters $(25,12,5,6),$ $10$ satisfy the hypothesis in Theorem \ref{thm:conference-h1}, while none satisfies the hypothesis in Corollary \ref{cor:Meshulam}. Of the $41$ conference graphs with parameters $(29,14,6,7)$, only the Paley$(29)$ satisfies the hypothesis in Theorem \ref{thm:conference-h1} and does not satisfy the condition in Corollary \ref{cor:Meshulam}. We also investigated the $6760$ conference graphs with parameters $(37,18,8,9)$ from Ted Spence's website \cite{Spence}. Of these, $6120$ satisfy the hypothesis in Theorem \ref{thm:conference-h1}, while none of the graphs in the entire collection satisfies the condition in Corollary \ref{cor:Meshulam}.

\section{Computations and examples}\label{sec:srgs}

In this section, we describe our computations and small examples of pairs of graphs which are not distinguished by the spectrum $L_1^{\uparrow}$. 

\subsection{Small strongly regular graphs}

In addition to the triangular graphs, we have performed computations on classes of strongly regular graphs on a small number of vertices, using SageMath \cite{sage}; the code is freely available here \cite{srgfunctions}. 
For the descriptions of the strongly regular graphs, we used \cite{Spence}.
In doing these computations, we find examples of classes of strongly regular graphs where every graph is distinguished by the spectrum of $L_1^{\uparrow}$, given in Table \ref{tab:label1}, and three pairs of co-parametric strongly regular graphs which are pairwise non-isomorphic and not distinguished by the spectrum of $L_1^{\uparrow}$, whose parameters are given in Table \ref{tab:label2}. We note that we did not perform an exhaustive search of all fully determined classes of strongly regular graphs up to $64$ vertices as the computation was costly; in particular, we did not test the class $(36,15,6,6)$ which contains $32\,548$ graphs (see \cite{McKaySpence2001}). 
\begin{table}[htbp]
    \centering
    \begin{tabular}{|l|l|c|}
    \hline
parameters & complement & number of graphs \\
\hline
(16, 6, 2, 2), & (16, 9, 4, 6) & 2 \\
(25, 12, 5, 6) & (25, 12, 5, 6) & 15 \\
(26, 10, 3, 4) & (26, 15, 8, 9) & 10 \\
(28, 12, 6, 4) & (28, 15, 6, 10) & 4 \\
(29, 14, 6, 7) & (29, 14, 6, 7) & 41 \\
(35, 18, 9, 9) & (35, 16, 6, 8) & 3854 \\
(36, 14, 4, 6) & (36, 21, 12, 12) & 180 \\
(45, 12, 3, 3) & (45, 32, 22, 24) & 78 \\
\hline
    \end{tabular}
    \caption{Classes of strongly regular graphs, which are distinguished by the spectrum of $L_1^{\uparrow}$.}
    \label{tab:label1}
\end{table}

\begin{table}[htbp]
    \centering
    \begin{tabular}{|l|l|c|}
\hline
parameters & complement & number of graphs \\
\hline
(40, 12, 2, 4) & (40, 27, 18, 18) & 28 \\
(64, 18, 2, 6) & (64, 45, 32, 30) & 167\\
\hline
    \end{tabular}
    \caption{Classes of strongly regular graphs, containing graphs which are not distinguished by the spectrum of $L_1^{\uparrow}$.}
    \label{tab:label2}
\end{table}

The smallest pair of cospectral strongly regular graphs consists of the Hamming graph $H(2,4)$ and the Shrikhande graph, both with parameters $(16,6,2,2)$. By Theorem \ref{thm:hammingn2}, the characteristic polynomials of the $L_1^{\uparrow}$ matrix of $H(2,4)$ is $x^{24}(x-4)^{24}$. By computer, we determined the corresponding polynomial for the Shrikhande graph which is $x^{17}(x-2)^9(x-4)^9(x-6)^1(x^2-6x+4)^{6}$.

There are three pairs of graphs, one pair with parameters $(40, 12, 2, 4)$ and two pairs with $(64, 18, 2, 6)$, which are cospectral with respect to  $L_1^{\uparrow}$, but not isomorphic. We note that all the complements were distinguished by the spectrum of $L_1^{\uparrow}$. The two graphs on $40$ vertices are both generalized quadrangles $\GQ(3,3)$ and their characteristic polynomial with respect to $L_1^{\uparrow}$ is:
\[x^{120}(x - 4)^{120}  
\]
which also follows from Lemma \ref{lem:gq}. 
The two pairs on $64$ vertices, have  characteristic polynomials with respect to $L_1^{\uparrow}$ as  follows:
\begin{align*}
(x - 6)^{4}  (x - 2)^{36}  (x - 4)^{228}  x^{260}  (x^{2} - 6x + 4)^{24}, \\
(x - 6)^{8}  (x - 2)^{72}  (x - 4)^{168}  x^{232}  (x^{2} - 6x + 4)^{48}.
\end{align*}
In each pair, there is one graph that is vertex-transitive while its $L_1^{\uparrow}$-mate is not.

In the next section, we will look more closely at graphs which are not distinguished by the spectrum of $L_1^{\uparrow}$. 

\subsection{Example of \texorpdfstring{$L_1^{\uparrow}$}{L1-up}-cospectral graphs}\label{sec:cosp}







\begin{figure}[htbp]
\centering
\includegraphics[scale=0.2]{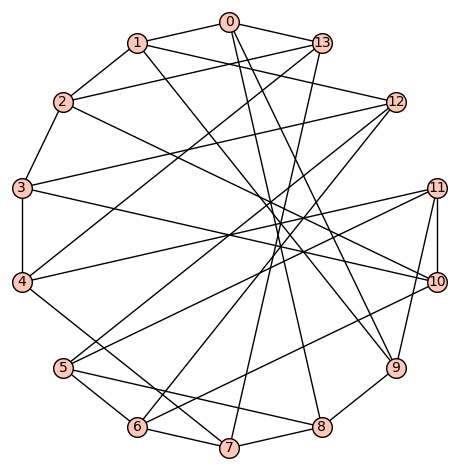}\includegraphics[scale=0.2]{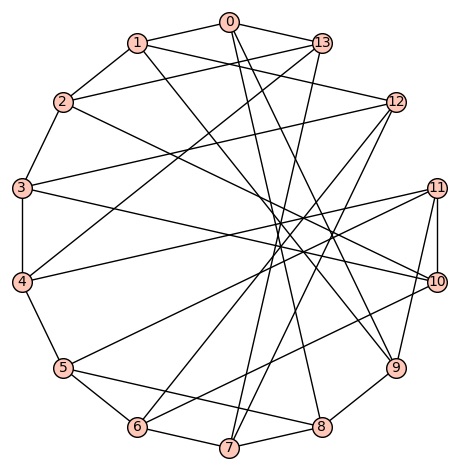}
\caption{Two graphs cospectral with respect to $L_0^{\uparrow}$ and $L_1^{\uparrow}$.}
\label{fig:4reg14}
\end{figure} 

By computer, we found two $4$-regular graphs on 14 vertices with \verb|graph6_strings| given below:
\[
\verb|MhCGGEDoHc@bSoiO?| \text{ and }
\verb|MhC?GUDoHc@bT_iO?| \]
which are non-isomorphic but cospectral with respect to $L_1^{\uparrow}$. These graphs are show in Figure \ref{fig:4reg14} and their characteristic polynomials are given below (the first for the usual adjacency matrix and the second for $L_1^{\uparrow}$):
\begin{equation*}
{ (x - 4) x (x + 3) (x^{11} + x^{10} - 15x^{9} - 13x^{8} + 81x^{7} +59x^{6} - 188x^{5} - 110x^{4} + 169x^{3} + 71x^{2} - 32x - 8),}
\end{equation*}
and 
\begin{equation*}
    {    x^{23}(x - 2)(x - 3)^{3}(x - 4)}.
\end{equation*}
We note that the complements of these two graphs are distinguished by the spectra of $L_1^{\uparrow}$. 

We also found four pairs of connected, non-isomorphic graphs $\{G_1,G_2\}$ on $6$ vertices, where $G_i, \overline{G_i}$ are not  triangle-free, $L_1^{\uparrow}(G_1), L_1^{\uparrow}(G_2)$ are cospectral and $L_1^{\uparrow}(\overline{G_1}), L_1^{\uparrow}(\overline{G_2})$ are cospectral (see Figure \ref{fig:cospec6}). We note that none of these pairs are cospectral with respect to the usual adjacency matrix. 

\begin{figure}[htbp]
    \centering
    \includegraphics[scale=0.9]{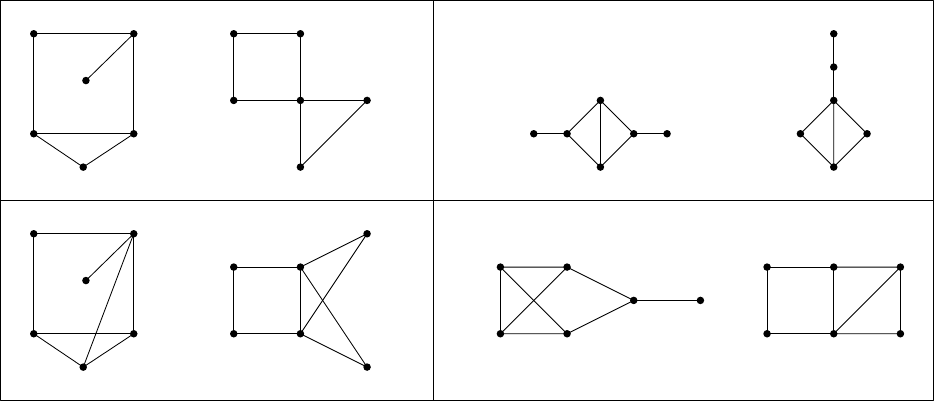}
    \caption{Four pairs of graphs which have the same $L_1^{\uparrow}$ spectrum and whose complements have the same $L_1^{\uparrow}$ spectrum.  \label{fig:cospec6}}
    
\end{figure}





\section{Conclusions and further directions}

In this paper, we studied the spectra of the up-Laplacians $L_1^{\uparrow}$ of the clique complexes of strongly regular graphs. Our motivation is rooted in the area of spectral characterization of graphs and we present examples of non-isomorphic strongly regular graphs with same parameters (and same ordinary Laplacian eigenvalues) with different $L_1^{\uparrow}$ spectrum. We also determine examples of non-isomorphic strongly regular graphs with same parameters with the same $L_1^{\uparrow}$ spectrum. However, we have not found any pair of cospectral strongly regular graphs (with the same parameters) which are cospectral with respect to $L_1^{\uparrow}$ and whose complements are also cospectral with respect to $L_1^{\uparrow}$. For a strongly regular graph $G$ with complement $\overline{G}$, the spectra of $L_1^{\uparrow}(G)$ and $L_1^{\uparrow}(\overline{G})$ determine the parameters of $G$. While it is tempting to conjecture that the spectra of $L_1^{\uparrow}(G)$ and $L_1^{\uparrow}(\overline{G})$ distinguish strongly regular graphs, this problem deserves more analysis and we plan to investigate it further in the future.

The computation of the spectrum of the Triangular graphs seems more involved the computation of the ordinary spectra and it would be nice to have a simpler way to arrive at these results. While the ordinary Laplacian (or adjacency matrix) spectrum of a primitive strongly regular graphs consists of three distinct eigenvalues that are the eigenvalues of the $3\times 3$ quotient matrix of the distance partition from a vertex, there does not seem to be such result for the higher order Laplacian matrices. The Triangular graph has five distinct $L_1^{\uparrow}$ eigenvalues, but we found other examples of graphs with more than five distinct eigenvalues. It is natural to try to understand how large can the number of distinct $L_1^{\uparrow}$ eigenvalues of a strongly regular graph be (as a function of the order or the size of the graph)? One can calculate the number of edges and the number of triangles from the spectrum of $L_1^{\uparrow}$. It would be interesting to determine other connections between the graph structure and its $L_1^{\uparrow}$ eigenvalues. Theorem \ref{prop:cycle-decomp} gives a sufficient condition for the first cohomology group to be trivial and we used this result for various families of graphs. We plan to investigate the tightness of our results as well as their applicability to other graphs in a future work.

\section*{Acknowledgments}

The first author thanks Alan Lew and Colby Sherwood for their comments and suggestions.

\bibliographystyle{plain}













\begin{appendices}

\section{The Invariance of the Laplacian Spectrum under Vertices Reordering}\label{a:inv}

In this appendix, we include a proof that the spectra of $L_s^{\uparrow}, L_s^{\downarrow},$ and $L_s$  do not depend upon the choice of the ordering for the elements of the ground set. We have not found such proof explicitly given in the literature and, for the sake of completeness, we  include a proof here.

\begin{lemma}\label{reordering}
    The spectra of $L_s^{\uparrow}, L_s^{\downarrow},$ and $L_s$ are invariant under reordering of the elements of the ground set for each $s$.
\end{lemma}
\begin{proof}
    It suffices to prove the lemma for $L_s^{\uparrow}$. Since the symmetric group $S_n$ is generated by the transpositions $(i,i+1)$ for $1 \leq i \leq n-1$, it suffices to show the invariance under reordering of two consecutive elements. That is, if the initial ordering is $1 < 2< \dots < n$, then let the new ordering be $1 < 2 < \dots < i-1 < i+1 < i < i+2 < \dots < n$ for some $1 \leq i \leq n-1$.
    Let $\delta_i$ and $\delta_i'$ denote the coboundary matrices under the first and the second orderings mentioned above respectively. Let $A,B,C$ denote the sets of $i$ faces that contain exactly $0,1,$ and $2$ of the elements $i,(i+1)$. Similarly, let $D,E,F$ be the sets of $i+1-$ faces that contain exactly $0,1,$ and $2$ of the elements $i,(i+1)$. In block form, $\delta_i$ has certain $0-$ blocks:
     \[
        \delta_s = \begin{bNiceMatrix}[first-row,first-col]
             & A & B & C \\
            D & a & 0 & 0 \\
            E & b & c & 0 \\
            F & 0 & d & e \\
        \end{bNiceMatrix},
    \]
    since if $U \in X_{i+1}, V \in X_i$ with $V \subset U$, then $|V\cap \{i,i+1\}| \leq |U \cap \{i,i+1\}| \leq |V\cap \{i,i+1\}|+1.$ Now, in $\delta_s',$ every block except for $d$ is the same as in $\delta_s$, while the $d$ block is replaced by $-d$ in $\delta_s'$. That is, 
    \[
        \delta_s' = \begin{bNiceMatrix}[first-row,first-col]
             & A & B & C \\
            D & a & 0 & 0 \\
            E & b & c & 0 \\
            F & 0 & -d & e \\
        \end{bNiceMatrix}.
    \]
    Now, consider the diagonal matrix $S \in \mathbf{R}_{X_{i+1} \times X_{i+1}}$ a with the diagonal entries on the rows corresponding to elements in $F$ are $-1$, and the remaining nonzero entries are $+1$. Similarly, let $T \in \mathbf{R}_{X_{i+1} \times X_{i+1}}$ be the diagonal matrix with the diagonal entries on columns corresponding to elements in $C$ are $-1$, and the remaining nonzero entries $+1$. We then have
    \begin{gather*}
        \delta_s' = S \delta_s T.
    \end{gather*}
    In addition, $SS^T = S^TS=I$ and $TT^T = T^TT=I.$
    Thus, if $L_s'^{\uparrow}$ is the upper-Laplacian under the new ordering,
    \begin{gather*}
        L_s'^{\uparrow} = \delta_s'^T \delta_s'
 = T^T\delta_s^TS^TS\delta_sT=T^{-1}L_s^{\uparrow}T.
    \end{gather*}
    Thus, $L_s'^{\uparrow}$ and $L_s^{\uparrow}$ are similar and hence, have the same spectrum.
\end{proof}

\section{Remaining details of \texorpdfstring{$n-$}{n}eigenvector constructions (Type 2) }\label{appendix:e_n_type2}
(Continuing with the notations and terminologies in the Type 2 case of \ref{eig_n}) 
\begin{itemize}
    \item  $e = \{\{b,i\},\{b,n\}\}$. Using \eqref{eq:verB}, we get that 
    \begin{gather*}
        (L_1^{\uparrow}w)_e=\sum_{p=0:p \not \in \{a,b\}}^{n-1} \sum_{f \in S_p}L_1^{\uparrow}(e,f) [S_p:f] \\
         =L_1^{\uparrow}(e,\{\{a,b\},\{b,n\}\})[S_0:\{\{a,b\},\{b,n\}]+L_1^{\uparrow}(e,\{\{b,i\},\{i,n\})[S_i:\{\{b,i\},\{i,n\}\}]\\
         =-L_1^{\uparrow}(e,\{\{a,b\},\{b,n\}\})+L_1^{\uparrow}(e,\{\{b,i\},\{i,n\})\\
         =[e:\{b,n\}][\{\{a,b\},\{b,n\}\}:\{b,n\}]-[e:\{b,i\}][\{\{b,i\},\{i,n\}\}:\{b,i\}]\\
         =(-1)(-1)-1\cdot 1=0=nw_e.
        \end{gather*}
        
        \item  $e = \{\{a,b\},\{b,i\}\}$. Using \eqref{eq:verB}, we get that 
        \begin{gather*}
        (L_1^{\uparrow}w)_e=\sum_{p=0:p \not \in \{a,b\}}^{n-1} \sum_{f \in S_p}L_1^{\uparrow}(e,f) [S_p:f] \\
        =L_1^{\uparrow}(e,\{\{a,b\},\{b,n\}\})[S_0:\{\{a,b\},\{b,n\}\}]+L_1^{\uparrow}(e,\{\{a,i\},\{b,i\}\})[S_i:\{\{a,i\},\{b,i\}\}]\\
        =-L_1^{\uparrow}(e,\{\{a,b\},\{b,n\}\})+L_1^{\uparrow}(e,\{\{a,i\},\{b,i\}\})\\
        =[e:\{a,b\}][\{\{a,b\},\{b,n\}\}:\{a,b\}]-[e:\{b,i\}][\{\{a,i\},\{b,i\}\}:\{b,i\}]\\
        =[e:\{a,b\}]+[e:\{b,i\}]=0=nw_e.
        \end{gather*}  
        
        \item  $e = \{\{a,n\},\{i,n\}\}$. Using \eqref{eq:verB}, we get that 
        \begin{gather*}
        (L_1^{\uparrow}w)_e=\sum_{p=0:p \not \in \{a,b\}}^{n-1} \sum_{f \in S_p}L_1^{\uparrow}(e,f) [S_p:f] \\
        =L_1^{\uparrow}(e,\{\{a,n\},\{b,n\}\})[S_0:\{\{a,n\},\{b,n\}\}]+L_1^{\uparrow}(e,\{\{i,n\},\{i,a\})[S_i:\{\{i,n\},\{i,a\}\}]\\
        =L_1^{\uparrow}(e,\{\{a,n\},\{b,n\}\})-L_1^{\uparrow}(e,\{\{i,n\},\{i,a\})\\
        =-[e:\{a,n\}][\{\{a,n\},\{b,n\}\}:\{a,n\}]+[e:\{i,n\}][\{\{i,a\},\{i,n\}\}:\{i,n\}]\\
        =-[e:\{a,n\}]-[e:\{i,n\}]=0=nw_e.
        \end{gather*}
        
        \item $e = \{\{b,n\},\{i,n\}\}$. Using \eqref{eq:verB}, we have that 
        \begin{gather*}
        (L_1^{\uparrow}w)_e=\sum_{p=0:p \not \in \{a,b\}}^{n-1} \sum_{f \in S_p}L_1^{\uparrow}(e,f) [S_p:f] \\
        =L_1^{\uparrow}(e,\{\{a,n\},\{b,n\}\})[S_0:\{\{a,n\},\{b,n\}\}]+L_1^{\uparrow}(e,\{\{i,n\},\{i,b\}\})[S_i:\{\{i,n\},\{i,b\}\}]\\
        =L_1^{\uparrow}(e,\{\{a,n\},\{b,n\}\})+L_1^{\uparrow}(e,\{\{i,n\},\{i,b\}\})\\
        =-[e:\{b,n\}][\{\{a,n\},\{b,n\}\}:\{b,n\}]-[e:\{i,n\}][\{\{i,n\},\{i,b\}\}:\{i,b\}]\\
        =[e:\{b,n\}]+[e:\{i,n\}]=0=nw_e.
        \end{gather*}     
\end{itemize}

\section{Remaining details of \texorpdfstring{$n-$}{n}eigenvector constructions (Type 3)}\label{appendix:e_n_type3}

\item (Continuing with the notations and terminologies in the Type 3 case of \ref{eig_n})

\begin{itemize}
    \item \textbf{Case 2:} $e = \{\{i,b\},\{i,x\}\}$. Using \eqref{eq:verB}, we get that
    \begin{gather*}
        (L_1^{\uparrow}w)_e=\sum_{p=0:p \not \in \{a,b\}}^{n-1} \sum_{f \in S_p}L_1^{\uparrow}(e,f) [S_p:f] \\ 
        =L_1^{\uparrow}(e,\{\{i,b\},\{i,n\}\})[S_i:\{\{i,b\},\{i,n\}\}]+L_1^{\uparrow}(e,\{\{i,a\},\{i,b\}\})[S_i:\{\{i,a\},\{i,b\}\}] \\
        =L_1^{\uparrow}(e,\{\{i,b\},\{i,n\}\})+ L_1^{\uparrow}(e,\{i,a\},\{i,b\}\}) \\
        =[e:\{i,b\}][\{\{i,b\},\{i,n\}\}:\{i,b\}]+[e:\{i,b\}][\{\{i,a\},\{i,b\}\}:\{i,b\}]\\
        =[e:\{i,b\}]-[e:\{i,b\}]=0=nw_e.
    \end{gather*}
    \item \textbf{Case 3:} $e = \{\{i,n\},\{i,x\}\}.$ Using \eqref{eq:verB} again,
    \begin{gather*}
        (L_1^{\uparrow}w)_e = \sum_{p=0:p \not \in \{a,b\}}^{n-1} \sum_{f \in S_p}L_1^{\uparrow}(e,f) [S_p:f] \\ 
        =L_1^{\uparrow}(e,\{\{i,n\},\{i,a\}\})[S_i:\{\{i,n\},\{i,a\}\}]+L_1^{\uparrow}(e,\{\{i,n\},\{i,b\}\})[S_i:\{\{i,n\},\{i,b\}\}]\\
        =-L_1^{\uparrow}(e,\{\{i,n\},\{i,a\}\})+L_1^{\uparrow}(e,\{\{i,n\},\{i,b\}\})\\
        =[e:\{i,n\}][\{\{i,n\},\{i,x\}\}:\{i,n\}]-[e:\{i,n\}][\{\{i,n\},\{i,b\}\}:\{i,n\}]\\
        =-[e:\{i,n\}]+[e:\{i,n\}]=0=nw_e.
    \end{gather*}
\end{itemize}

\end{appendices}
\end{document}